\documentclass[preprint,noinfoline,addressatend,imslayout]{imsart}
\usepackage[top=20mm, bottom=20mm, left=15mm, right=15mm]{geometry}

\usepackage{url}
\RequirePackage[colorlinks,citecolor=blue,urlcolor=blue]{hyperref}
\usepackage[authoryear]{natbib}

\usepackage{amsthm,amsmath,amsfonts,amssymb}
\usepackage{enumerate}
\usepackage{bm,accents}
\usepackage{graphicx}
\usepackage{nicefrac}
\usepackage{color}
\usepackage{nameref}

\newtheorem{thm}{Theorem}
\newtheorem{cor}{Corollary}
\newtheorem{lem}{Lemma}

\newtheorem{defn}{Definition}
\newtheorem{rem}{Remark}
\DeclareMathOperator*{\argmin}{argmin} 
\bibpunct{(}{)}{;}{a}{,}{,}
\def\spacingset#1{\renewcommand{\baselinestretch}%
{#1}\small\normalsize} \spacingset{1.6}

\begin{document}

\begin{frontmatter}


\title{Kernel Machines for Current Status Data}
\maketitle
\runtitle{Kernel Machines for Current Status Data}
\thankstext{T1}{The authors were funded in part by NSF grant DMS-1407732.}

\begin{aug}
\author{\fnms{Yael} \snm{Travis-Lumer}\ead[label=e1]{travis-lumer@campus.technion.ac.il}}
and \author{\fnms{Yair}
\snm{Goldberg}\ead[label=e2]{yairgo@technion.ac.il}}

\affiliation{Technion}

\address{Yael Travis-Lumer and Yair Goldberg\\The Faculty of Industrial Engineering and Management\\ Technion, Haifa 3200003, Israel\\ \printead{e1}\\
\phantom{E-mail:\ }\printead*{e2} }

\runauthor{Travis-Lumer and Goldberg}
\end{aug}
\begin{abstract}
In survival analysis, estimating the failure time distribution is an important and difficult task, since usually the data is subject to censoring. Specifically, in this paper we consider current status data, a type of data where all of the observations are censored. The format of the data is such that the failure time is restricted
	to knowledge of whether or not the failure time exceeds a random monitoring
	time. We propose a flexible kernel machine approach for estimation of the failure time expectation as a function
	of the covariates, with current status data. In order to obtain the kernel machine decision
	function, we minimize a regularized version of the empirical risk
	with respect to a new loss function. Using finite sample bounds and novel oracle inequalities,
	we prove that the obtained estimator converges to the true
	conditional expectation for a large family of probability measures. Finally, we present a simulation
	study and an analysis of real-world data that compares the performance of the proposed approach to existing methods. We show empirically that our approach is comparable to current
	state of the art, and in some cases is even better.
\end{abstract}
\end{frontmatter}

\section{\label{sec:Introduction}Introduction}

In this paper we aim to develop a general model free method for
analyzing current status data using machine learning techniques. In
particular, we propose a kernel machine learning method
for estimation of the failure time expectation with current status
data. Kernel machines, also known as support vector machines, were originally introduced by Vapnik in the 1990's and are
firmly related to statistical learning theory \citep{vapnik_nature_1999}. Kernel machines are learning algorithms that utilize positive definite kernels \citep{hofmann_kernel_2008}. 
The choice of kernel machines for current status data is motivated by the fact
that kernel machines can be implemented easily, have fast training speed, produce
decision functions that have a strong generalization ability, and can
guarantee convergence to the optimal solution, under some weak assumptions
\citep{shivaswamy_support_2007}.

The format of current status data is such that the failure time $T$ is
restricted to knowledge of whether or not $T$ exceeds a random monitoring
time $C$. This data format is quite common and includes examples
from various fields. \citet{jewell_current_2002} mention a few examples
including: studying the distribution of the age of a child at weaning
given observation points; when conducting a partner study of HIV infection
over a number of clinic visits; and when a tumor under investigation
is occult and an animal is sacrificed at a certain time point in order
to determine presence or absence of the tumor. For instance, in the
last example, when performing carcinogenicity testing, $T$ is the time from exposure
to a carcinogen and until the presence of a tumor, and $C$ is the
time point at which the animal is sacrificed in order to determine
presence or absence of the tumor. Clearly, it is difficult to estimate
the failure time distribution since we cannot observe the failure
time $T$. These examples illustrate the importance of this topic
and the need to find advanced tools for analyzing such data.

There are several approaches for analyzing current status data. Traditional
methods include parametric models where the underlying distribution
of the survival time is assumed to be known, such as Weibull, Gamma,
and other distributions with non-negative support. Other approaches
include semiparametric models, such as the Cox proportional hazards
model, and the accelerated failure time (AFT) model \citep[see, for example, ][]{klein_survival_1997}. Several works including \citet{diamond_proportional_1986}, \citet{shiboski_statistical_1992},
\citet{jewell_current_2002} and others, have suggested the Cox proportional
hazard model for current status data, where the Cox model can be represented
as a generalized linear model with a log-log link function. Other
works, including \citet{tian_accelerated_2006}, discussed the use of
the AFT model for current status data and suggested different algorithms
for estimating the model parameters. Additional semiparametric regression models for current status data include proportional odds \citep{rossini_semiparametric_1996},  additive hazards \citep{lin_additive_1998}, additive transformations \citep{cheng_semiparametric_2011}, linear transformations \citep{sun_semiparametric_2005}, and linear regression \citep{shen_linear_2000}. Needless to say that both parametric
and semiparametric models demand stringent assumptions on the distribution
of interest which can be restrictive. For this reason, additional
estimation methods are needed.

Nonparametric methods for analyzing current status data were also investigated in the literature. Nonparametric maximum likelihood estimation (NPMLE) of the failure time distribution function is commonly used with this type of data, and relies on the PAV algorithm of \cite{ayer_empirical_1955}. \cite{burr_2002} studied nonparametric estimation of the conditional distribution function of the failure time given the covariates, based on a locally smoothed modification of the NPMLE. \citet{wang_nonparametric_2012} studied nonparametric estimation of the marginal distribution function of the failure time using the copula model approach. \cite{honda_nonparametric_2004} constructed an estimator for the regression function utilizing a modification of maximum rank correlation, and estimated the difference between the regression function at some value, to the regression function at a standard fixed point. Note that these works are not specifically intended for estimation of the conditional expectation and thus might not yield accurate estimates.

Over the past two decades, some learning algorithms for censored data have been proposed. However, most of
these algorithms cannot be applied to current status data but only
to other, more common, censored data formats. Recently, several authors
suggested the use of kernel machines, or similarly support vector machines, for survival data, including \citet{van_belle_support_2007}, \citet{khan_support_2008}, \citet{eleuteri_support_2011}, \citet{goldberg_support_2017}, \citet{shiao_svm-based_2013}, \citet{wang_support_2016}, and \citet{polsterl_efficient_2016}. These examples illustrate that initial steps in this
direction have already been taken. However, as far as we know, the
only work based on kernel machines that can also be applied to current status data
is by \citet{shivaswamy_support_2007} which has a more computational
and less theoretic nature. The authors studied the use of kernel machines for
regression problems with interval censoring and, using simulations,
showed that the method is comparable to other missing data tools.

We present a kernel machine framework for current status data.
We propose a learning method, denoted by KM-CSD, for estimation of
the failure time conditional expectation. We investigate the theoretical properties
of the KM-CSD, and in particular, prove consistency for a large family
of probability measures. In order to estimate the conditional expectation
we use a modified version of the quadratic loss, using the methodology
of  \citet{van_der_laan_locally_1998,laan_unified_2003}. Since the failure time $T$ is not
observed, our new modified loss function is based on the censoring
time $C$ and on the current status indicator. Finally, in order to
obtain the KM-CSD estimator, we minimize
a regularized version of the empirical risk with respect to our new proposed
loss. Note that the terminology decision function is used in the kernel machine context to describe the obtained estimator.

The kernel machine we present in this work may be referred to as an inverse probability weighted complete-case estimator (\citealt{laan_unified_2003}, \citealt[Chapter 6]{tsiatis_semiparametric_2007}).  It is tempting to use the tools described in these books to derive doubly-robust kernel machine estimators. In the context of estimating equations with missing data, doubly-robust estimators are typically constructed by adding an augmentation term. This term is constructed by projecting the estimating equation onto the augmentation space (see \citealt[Section 7.4, and Theorem 10.1]{tsiatis_semiparametric_2007}). However, in our kernel machine setting, the estimator is obtained as the minimizer of a weighted loss function over a reproducing kernel Hilbert space (RKHS) and thus it is not clear how meaningful it is to project the loss function on the augmentation space. It is also not trivial to add a term to the proposed regularized empirical risk minimization problem in a way that yields a convex optimization problem over an RKHS, which is essential for deriving the results presented in this paper. While doubly-robust estimators for current status data were derived in the semiparametric literature (\citealt{ANDREWS2005332}), we do not consider such estimators in this work.  To the best of our knowledge, the only work that studied doubly-robust estimators  in the context of kernel machines was done by \cite{liu2018kernel}, however this was done in the context of missing responses, and cannot be applied to our case.  


The contribution of this work includes the development of a nonparametric estimator of the conditional expectation, the development of a kernel machine framework
for current status data, the development of new oracle inequalities for censored
data, and the study of the theoretical properties and the consistency of the
KM-CSD.

The paper is organized as follows. In Section \ref{sec:Preliminaries}
we describe the formal setting of current status data and discuss
the choice of the quadratic loss for estimating the conditional expectation.
In Section \ref{sec:Kernel-Machines} we present the proposed
KM-CSD and its corresponding loss function. Section
\ref{sec:Theoretical-Results} contains the main theoretical results,
including finite sample bounds and consistency. Section \ref{sec:Simulation-Study}
contains the simulations and Section \ref{sec:Case Study} contains an analysis of real world data. Concluding remarks are presented in Section
\ref{sec:Concluding-Remarks}. The proofs appear in \nameref{sec:Appendix}. The Matlab code for both the algorithm and for
the simulations, as well as the artificially censored data from Section \ref{sec:study2}, can be found in the \nameref{sub:Supplementary-Material}.

\section{\label{sec:Preliminaries}Preliminaries}

In this section we present the notation used throughout the paper.
First we describe the data setting and then we discuss briefly loss
functions and risks.

Assume that the data consists of $n$ independent and identically distributed random triplets $D=\{(Z_{1},C_{1},\Delta_{1}),\\
\ldots,\allowbreak (Z_{n},C_{n},\Delta_{n})\}$.
The random vector $Z$ is a vector of covariates that takes its values
in a compact set $\mathcal{Z}\subset\mathbb{\mathbb{R}}^{d}$. The
failure-time $T$ is non-negative, the random variable $C$ is the  non-negative
censoring time, where both $C$ and $T$ are contained in the
interval $[0,\tau]\equiv\mathcal{Y},$ for some constant $\tau>0$. The indicator $\Delta=\mathbf{{1}}\{T\leq C\}$ is
the current status indicator at time $C$, obtaining the value 1 when $T\leq C$, and 0 otherwise.
For example, in carcinogenicity testing, an animal is sacrificed at
a certain time point in order to determine presence or absence of
the tumor. In this example, $T$ is the time from exposure to a carcinogen
and until the presence of a tumor, $Z$ can be any explanatory information collected such as the weight of the animal, $C$ is the time point at which
the animal is sacrificed, and $\Delta$ is the current status indicator
at time $C$ (indicating whether the tumor has developed before the
censoring time, or not).

We now move to discuss a few definitions of loss functions and risks,
following \citet{steinwart_support_2008}. Let$\left(\mathcal{\mathcal{Z}},\mathcal{A}\right)$
be a measurable space and $\mathcal{Y}\subset\mathbb{\mathbb{R}}$
be a closed subset. Then a loss function is any measurable function
$L$ from $\mathcal{Y}\times\mathbb{R}$ to $[0,\infty)$.

Let $L:\mathcal{Y}\times\mathbb{R}\rightarrow[0,\infty)$
be a loss function and $P$ be a probability measure on $\mathcal{Z\times Y}$.
For a measurable function $f:\mathcal{Z}\mapsto\mathbb{\mathbb{R}}$,
the $L$-risk of $f$ is defined by $R_{L,P}\left(f\right)\equiv E_{P}\left[L\left(Y,f(Z)\right)\right]=\int_{Z\times Y}L\left(y,f(z)\right)dP(z,y)$.
A function $f$ that achieves the minimum $L$-risk is called a \textit{Bayes
	decision function} and is denoted by $f^{*}$, and the minimal $L$-risk
is called the \textit{Bayes risk} and is denoted by $R_{L,P}^{*}$.
Finally, the empirical $L$-risk is defined by $R_{L,D}\left(f\right)=\frac{1}{n}{\displaystyle \sum_{i=1}^{n}}L(y_{i},f(z_{i}))$. It is well known \citep[see, for example, ][]{hastie_elements_2013}
that the conditional expectation is the Bayes decision function with
respect to the quadratic loss. That is, $E[Y|Z]=f^{*}=\argmin_{f} R_{L,P}(f)$, where $L$ is the quadratic loss defined by $L(Y,f(Z))=(Y-f(Z))^2$.

Recall that our goal is to estimate the conditional expectation of the failure-time $T$ given the covariates $Z$. However, in the setting of current status data, the response variable (failure-time) is not observed, making the estimation procedure more complex. It is not even clear if and how loss functions can be defined with current status data. In the following section we construct a new modification of the quadratic loss that is based on the censoring time and on the current status indicator, and use it to estimate the conditional expectation of the unobservable failure-time. 
\section{\label{sec:Kernel-Machines}Kernel Machines for Current
	Status Data}
This section is divided into three subsections. We start by describing general kernel machines for uncensored data. Then we define a new loss function for current status data, utilizing an equality between risks, and incorporate it into the kernel machine framework. Finally we define the proposed estimator of the conditional expectation of the failure-time, with current status data, and discuss some assumptions regarding the censoring mechanism.


\subsection{Kernel Machines for Uncensored Data}
Let $\mathcal{H}$ be a reproducing kernel Hilbert space (RKHS) of
functions from $\mathcal{Z}$ to $\mathbb{R}$, where an RKHS is a
function space that can be characterized by some kernel function $k:\mathcal{Z}\times\mathcal{Z}\mapsto\mathbb{R}$. For more information on reproducing kernel Hilbert spaces, we refer the reader to \citealt[Chapter 4]{steinwart_support_2008}.
A continuous kernel $k$ for which the corresponding RKHS $\mathcal{H}$ is dense in the space of continuous functions on $\mathcal{Z}$, $C(\mathcal{Z})$, is called a universal kernel
(see, for example, \citealt[Definition 4.52]{steinwart_support_2008}).
Fix such an RKHS $\mathcal{H}$ and denote its norm by $\left\Vert \cdot\right\Vert _{\mathcal{H}}$.
Let $\{\lambda_{n}\}>0$ be some sequence of regularization constants.
A kernel machine decision function for uncensored data is defined by:
\[
f_{D,\lambda_{n}}=\arg\min{}_{f\in\mathcal{H}}\lambda_{n}\|f\|_{\mathcal{H}}^{2}+\frac{1}{n}\sum_{i=1}^{n}L(T_{i},f(Z_{i}))\,.
\]

\subsection{Equality Between Risks}

In this subsection we show that the risk can be represented as the sum of two terms
\begin{align*}
E\left[\frac{(1-\Delta)\ell(C,f(Z))}{g(C|Z)}\right]+E[L(0,f(Z))].
\end{align*}
We recall that current status data consists of $n$ independent and
identically-distributed random triplets $D=\{(Z_{1},C_{1},\Delta_{1}),\ldots,(Z_{n},C_{n},\Delta_{n})\}$.

Let $F(\cdot|Z=z)$ and $G(\cdot|Z=z)$ be the cumulative distribution
functions of the failure time and censoring, respectively, given the
covariates $Z=z$. Let $g(\cdot|Z=z)$ be the density of $G(\cdot|Z=z)$. 
Throughout this work we will assume the following:
\newenvironment{tight_enumerate}{
	\begin{enumerate}
		\setlength{\itemsep}{0pt}
		\setlength{\parskip}{0pt}
	}{\end{enumerate}}

\begin{enumerate}[\hspace{0.5cm}(1)]
	\renewcommand{\labelenumi}{(A\arabic{enumi})}
	\item   The censoring time $C$ is independent of the failure time
	$T$ given the covariates $Z$.\label{as:A1}
	\item  $C$ and $T$ take values in the interval $[0,\tau]\equiv\mathcal{Y}$ and $\underset{_{z\in\mathcal{Z},c\in \mathcal{Y}}}{\inf}g\left(c|z\right)\geq2\kappa>0$,
	for some $\kappa>0$. \label{as:A2}
\end{enumerate}
The conditional independence assumption~(A\ref{as:A1}) is a standard identifiability assumption in survival analysis (see, for example, \citealt{klein_survival_1992} and \citealt{klein_survival_1997}). Assumption~(A\ref{as:A2}) is needed in order to guarantee that integration with respect to $T$ and $C$ can be exchanged, and in order to allow for division by the censoring density. Similar assumptions were made by \citet{van_der_laan_locally_1998}.

For current status data, we introduce the following identity between
risks, following \citet{van_der_laan_locally_1998,laan_unified_2003}. We extend this
identity by incorporating loss functions and covariates. Let $L:\mathcal{Y}\times\mathbb{\mathbb{R}}\mapsto[0,\infty)$
be a loss function differentiable in the first variable. Let $\ell:\mathcal{Y}\times\mathbb{\mathbb{R}}\mapsto\mathbb{\mathbb{R}}$
be the derivative of $L$ with respect to the first variable.

We would like to find the minimizer of $R_{L,P}(f)$ over a set $\mathcal{H}$
of functions $f$. Note that
\begin{align*}
R_{L,P}(f)\equiv & E_{Z}E_{T|Z}L(T,f(Z))
=  E_{Z}\left[\int_{0}^{\tau}L(t,f(Z))dF(t|Z)\right]\\
= & E_{Z}\left[\int_{0}^{\tau}\ell(t,f(Z))(1-F(t|Z))dt-\left.L(t,f(Z))(1-F(t|Z))\right|_{0}^{\tau}\right]\\
= & E_{Z}\left[\int_{0}^{\tau}\ell(t,f(Z))(1-F(t|Z))dt\right]+E[L(0,f(Z))]\,,
\end{align*}
and that $(1-\Delta)=\mathbf{{1}}\{T>C\}$ and thus
\begin{align*}
E\left[\frac{(1-\Delta)\ell(C,f(Z))}{g(C|Z)}\right]= & E_{Z,T}\left[E_{C}\left[\left.\frac{\mathbf{{1}}\{T>C\}\ell(C,f(Z))}{g(C|Z)}\right|Z,T\right]\right]\\
= & E_{Z,T}\left[\int_{0}^{\tau}\frac{\mathbf{{1}}\{T>c\}\ell(c,f(Z))g(c|Z)}{g(c|Z)}dc\right]\\
= & E_{Z,T}\left[\int_{0}^{\tau}\mathbf{{1}}\{T>c\}\ell(c,f(Z))dc\right]\\
= & E_{Z}\left[\int_{0}^{\tau}\ell(c,f(Z))\int_{0}^{\tau}\mathbf{{1}}\{t>c\}dF(t|Z)dc\right]\\
= & E_{Z}\left[\int_{0}^{\tau}\ell(c,f(Z))(1-F(c|Z))dc\right]\,.
\end{align*}
\subsection{Kernel Machines for Current Status Data}
Hence, in order to estimate the minimizer of $R_{L,P}(f)$, one can
minimize a regularized version of the empirical risk with respect
to a new loss function defined by
\[
L^{n}(D,(Z,C,\Delta,s))=\frac{(1-\Delta)\ell(C,s)}{g(C|Z)}+L(0,s)\,.
\]
Note that this function need not be convex nor a loss function. Recall that we are interested in estimating the conditional expectation. This means that we would like to minimize the risk with respect to the quadratic loss. For
the quadratic loss, our new loss function becomes
\[
L^{n}(D,(Z,C,\Delta,s))=\frac{(1-\Delta)2(C-s)}{g(C|Z)}+s^{2}\,.
\]

Note that this function is convex but not necessarily a loss function
since it can obtain negative values. However, one can always add a constant to ensure positivity. Since this constant does not effect optimization it will be neglected hereafter. For a detailed explanation, see Appendix~\ref{positivity}. 

In order to implement this result into the kernel machine framework, we propose
to define the KM-CSD decision function for current status data by
\begin{equation}
f_{D,\lambda}=\arg\min{}_{f\in\mathcal{H}}\lambda\|f\|_{\mathcal{H}}^{2}+\frac{1}{n}{\displaystyle \sum_{i=1}^{n}}\left[\frac{(1-\Delta_{i})2(C_{i}-f(Z_{i}))}{g(C_{i}|Z_{i})}+(f(Z_{i}))^{2}\right].\label{eq: KM-CSD}
\end{equation}
Explicit computation of the decision function can be found in Appendix~\ref{closed-form}. Note that if the censoring mechanism is unknown, we can replace
the density $g$  in Eq.~\ref{eq: KM-CSD} with its estimate $\hat{g}$, as long as $\hat{g}$ is strictly positive on $[0,\tau]\equiv\mathcal{Y}$; in this case the kernel machine decision function is
\begin{align*}
f_{D,\lambda}=\arg\min{}_{f\in\mathcal{H}}\lambda\|f\|_{\mathcal{H}}^{2}+\frac{1}{n}{\displaystyle \sum_{i=1}^{n}}\left[\frac{(1-\Delta_{i})2(C_{i}-f(Z_{i}))}{\hat{g}(C_{i}|Z_{i})}+(f(Z_{i}))^{2}\right]
\end{align*}
(note the use of $\hat{g}$ instead of $g$ in the denominator).

We note that for current status data, the assumption of some knowledge
of the censoring distribution is reasonable, for example, when it
is chosen by the researcher \citep{jewell_current_2002}. In other
cases, the density can be estimated using either parametric or nonparametric density estimation
techniques such as kernel estimates. It should be noted that the censoring
variable itself is fully observed (not censored) and thus simple density estimation
techniques can be used in order to estimate the density $g$.


\section{\label{sec:Theoretical-Results}Theoretical Results}

The main goal of our work is to find a `good' estimator of the failure time conditional expectation. A good estimator should first and foremost be consistent, that is, its risk should converge in probability to the Bayes risk. Additionally, we would like such an estimator to be consistent for a large family of probability measures. 
The consistency proof is based on novel oracle inequalities that are presented
below.

We start by proving risk consistency of the KM-CSD learning method
for a large family of probability measures. We first assume that the censoring mechanism is known, which means that the true density of the censoring variable $g$ is known. Using
this assumption, and some additional conditions, we bound the difference
between the risk of the KM-CSD decision function and the Bayes risk
in order to form finite sample bounds. We use this result to show that the KM-CSD converges in probability
to the Bayes risk. That is, we demonstrate that for a large family
of probability measures, the KM-CSD learning method is consistent.
We then consider the case in which the censoring mechanism is unknown, and thus the density $g$ needs to be estimated. We estimate
the density $g$ using nonparametric kernel density estimation, and
develop a novel finite sample bound. We use this bound to prove that
the KM-CSD is consistent even when the censoring distribution is
unknown. 

For simplicity, we use the normalized version of the quadratic loss.
\begin{defn}
	\label{def: normalized loss}Let $L(y,s)=\frac{(y-s)^{2}}{\tau^{2}}$
	be the normalized quadratic loss, let $l(y,s)=\frac{2(y-s)}{\tau^{2}}$
	be its derivative with respect to the first variable, and let $L^{n}(D,(Z,C,\Delta,s))=\frac{1}{\tau^{2}}\left(\frac{(1-\Delta)2(C-s)}{g(C|Z)}+s^{2}\right)$
	be the proposed modified version of this loss.
\end{defn}

Since both $L$ and $l$ are convex functions with respect to $s$,
then for any compact set $\mathcal{S}=[-S,S]\subset\mathbb{R}$, Both
$L$ and $l$ are bounded and Lipschitz continuous with constants
$c_{L}$ and $c_{l}$ that depend on $\mathcal{S}$.
\begin{rem}
	\label{rem:loss bounds}$L(y,0)\leq1$ for all $y\in\mathcal{Y}$
	and $\ell(y,s)\leq B_{1}$ for all $(y,s)\in\mathcal{Y\times S}$
	and for some constant $B_{1}>0$.
\end{rem}
We need the following additional assumptions:
\begin{enumerate}[\hspace{0.5cm}(1)]
	\addtocounter{enumi}{2}
	\renewcommand{\labelenumi}{(A\arabic{enumi})}
	\item  $\mathcal{Z}\subset\mathbb{\mathbb{R}}^{d}$ is compact, \label{as:A3}
	\item  $\mathcal{H}$ is an RKHS of a continuous kernel $k$ with
	$\left\Vert k\right\Vert _{\infty}\le1$. \label{as:A4}
\end{enumerate}
Assumptions~(A\ref{as:A3}-A\ref{as:A4}) are standard technical assumptions in the kernel machines literature. 

Define the approximation error by $A_{2}(\lambda)=\underset{f\in\mathcal{H}}{\inf}\lambda\left\Vert f\right\Vert _{\mathcal{H}}^{2}+R_{L,P}(f)-R_{L,P}^{*}$. Define $B_{2}=c_{L}\lambda^{-\nicefrac{1}{2}}+1$ and $B=\frac{B_{1}}{2\kappa}+B_{2}$,
where $B_{1}$ is defined in Remark~\ref{rem:loss bounds}, $\kappa$ is defined in Assumption~(A\ref{as:A2}), $c_{L}$ is the Lipschitz constant of the normalized quadratic loss $L$, and $\lambda$ is the regularization parameter.

\subsection{Case I - The Censoring Density $g$ is Known}

In this section we develop finite sample bounds assuming that the
censoring density $g$ is known.
\begin{thm}
	\label{ theorem 1 - g known}Assume that (A\ref{as:A1})-(A\ref{as:A4}) hold. Then for
	fixed $\lambda>0,\, n\ge1,\,\varepsilon>0$, and $\theta>0$, with
	probability not less than $1-\exp(-\theta)$
\[
	   \lambda\left\Vert f_{D,\lambda}\right\Vert _{\mathcal{H}}^{2}+R_{L,P}(f_{D,\lambda})-R_{L,P}^{*}-A_{2}(\lambda)\\
	 \leq B\sqrt{\frac{2\log\left(2N(\sqrt{\frac{1}{\lambda}}B_{H},\left\Vert \cdot\right\Vert _{\infty},\epsilon)\right)+2\theta}{n}}+\frac{2c_{l}\varepsilon}{\kappa}+4c_{L}\varepsilon
	\]
	where $N(\lambda^{-\frac{1}{2}}B_{H},\left\Vert \cdot\right\Vert _{\infty},\epsilon)$
	is the covering number of the $\varepsilon-net$ of $\sqrt{\frac{1}{\lambda}}\overline{B_{H}}$
	with respect to supremum norm and where $\overline{B_{H}}$ is the closure of the unit ball
	of $\mathcal{H}$ (for further details see \citealt{steinwart_support_2008})
	.
\end{thm}
The proof of this theorem appears in Appendix \ref{sub:Proof-of-Theorem 1}.

We now move to discuss consistency of the KM-CSD learning method.
By definition, $P$-universal consistency means that for any $\epsilon>0$,

\begin{equation}
\underset{n\rightarrow\infty}{\lim}P(D\in(\mathcal{Z\times\mathcal{Y}})^{n}\,:\,R_{L,P}(f_{D,\lambda_{n}})\leq R_{L,P}^{*}+\epsilon)=1,\label{eq:3}
\end{equation}
where $R_{L,P}^{*}$ is the Bayes risk. Universal consistency
means that (\ref{eq:3}) holds for all probability measures $P$ on
$\mathcal{Z\times Y}$. However, in survival analysis we have the
problem of identifiability and thus we will limit our discussion to
probability measures that satisfy some identification conditions.
Let $\mathcal{P}$ be the set of all probability measures that satisfy
Assumptions (A\ref{as:A1})-(A\ref{as:A2}). We say that a learning method is $\mathcal{P}$-universal
consistent when (\ref{eq:3}) holds for all probability measures $P\in\mathcal{P}$.

In order to show $\mathcal{P}$-universal consistency, we utilize
the finite sample bounds of Theorem~\ref{ theorem 1 - g known}. The
following assumption is also needed for proving $\mathcal{P}$-universal
consistency:
\begin{enumerate}[\hspace{0.5cm}(1)]
	\addtocounter{enumi}{4}
	\renewcommand{\labelenumi}{(A\arabic{enumi})}
	\item  \label{as:A5} $k$ is a universal kernel. 
\end{enumerate}
Universal kernels are a wide family of kernel functions that include Gaussian and Taylor kernels. A kernel $k$ is called universal if the RKHS $\mathcal{H}$ of $k$
is dense in the space of continuous functions on $\mathcal{Z}$, $C(\mathcal{Z})$, with respect to the sup norm. Assumption~(A\ref{as:A5}) means that $\underset{f\in\mathcal{H}}{\inf}R_{L,P}(f)=R_{L,P}^{*}$,
for all probability measures $P$ on $\mathcal{Z\times Y}$. 
\begin{cor}
	\label{cor - consistency} Assume the setting of Theorem~\ref{ theorem 1 - g known}
	and that Assumption~(A\ref{as:A5}) holds. Assume that there exist constants $a\geq 1$ and $p>0$ such that $\log\left(N(B_{H},\left\Vert \cdot\right\Vert _{\infty},\epsilon)\right)\leq a\epsilon^{-2p}$. Let $\lambda_{n}$ be a sequence such
	that $\lambda_{n}\underset{n\rightarrow\infty}{\rightarrow}0$ and
	$\lambda_{n}^{1+p}n\underset{n\rightarrow\infty}{\rightarrow}\infty.$  Then the KM-CSD learning
	method is $\mathcal{P}$-universal consistent. \end{cor}
	The proof of this theorem appears in Appendix~\ref{sub:Proof-of-cor 1}.
	
	Note that the bound on the covering number $N(B_{H},\left\Vert \cdot\right\Vert _{\infty},\epsilon)$ in Corollary~\ref{cor - consistency} is satisfied for smooth kernels, such as polynomial and Gaussian
kernels, for arbitrarily small $p>0$ (see \citealt[Section 6.4]{steinwart_support_2008}).

\subsection{Case II - The Censoring Density $g$ is Unknown}
	Here we consider the case in which the censoring mechanism is unknown, and thus the density $g$ needs to be estimated. We estimate
	the density $g$ using nonparametric kernel density estimation, and
	develop a novel finite sample bound. We use this bound to prove that
	the KM-CSD is consistent even when the censoring distribution is
	unknown. Note that asymptotic results for kernel density estimators are well known in the literature (see, for example, \citealt{silverman1978}). However, to the best of our knowledge, finite sample bounds for this case do not exist and hence are developed here.
	
	For simplicity, we assume here that the censoring time $C$ is independent of the covariates $Z$. One can generalize the estimation procedure to include dependence of the censoring time $C$ on the covariates $Z$; for example, the conditional density estimate can be computed by the ratio of the joint density	estimate to the marginal density estimate.  In Lemma~\ref{lemma on density} we construct
	finite sample bounds on the difference between the estimated density
	$\hat{g}$ and the true density $g$. In Theorem~\ref{theorem g unknown}
	we utilize this bound to form finite sample bounds for the KM-CSD
	learning method.
	\begin{defn}
		\label{kernel of order m}We say that $K_m:\mathbb{\mathbb{R}\mapsto R}$
		(not to be confused with the kernel function $k$ of the RKHS $\mathcal{H})$
		is a kernel of order $m$, if the functions $u\mapsto u^{j}K_m(u)\,\,,j=0,1,...,m$
		are integrable and satisfy $\int_{-\infty}^{\infty}K_m(u)du=1$
		and $\int_{-\infty}^{\infty}u^{j}K_m(u)du=0,\,\, j=1,...,m$.
	\end{defn}
	
	\begin{defn}\label{Holder class}
		The H\"{o}lder class $\sum(\beta,\mathcal{L})$
		of functions $f:\mathbb{\mathbb{R}\mapsto R}$ is the set of $m=\left\lfloor \beta\right\rfloor $
		times differentiable functions whose derivative $f^{(m)}$ satisfies
		$\left|f^{(m)}(x)-f^{(m)}(x')\right|\leq\mathcal{L}\left|x-x'\right|^{\beta-m}$
		for any $x, x' \in \mathbb{R}$ and for some constant $\mathcal{L}>0$.\end{defn}
	\begin{lem}
		\label{lemma on density}Let $K_m:\mathbb{\mathbb{R}\mapsto R}$ be
		a kernel function of order  $m=\left\lfloor \beta\right\rfloor $ satisfying $\int_{-\infty}^{\infty}K_{m}^{2}(u)du<\infty$
		and define $\hat{g}(x)=(hn)^{-1}\stackrel{}{\sum_{i=1}^n K_m\left((C_{i}-x)/h\right)}$ where
		$h$ is the bandwidth. Suppose that the true density $g$ and its estimate $\hat{g}$ both satisfy
		$g(c),\hat{g}(c)\leq g_{max}<\infty$. Let us also assume that $g(c)$ belongs
		to the $\mathrm{H\ddot{o}lder}$ class $\sum(\beta,\mathcal{L})$.
		Finally, assume that $\int_{-\infty}^{\infty}\left|u\right|^{\beta}\left|K_m(u)\right|du<\infty$.
		Then for any $\theta>0$,
		\[
		Pr\left(\frac{1}{n}\sum_{i=1}^n \left|\hat{g}(C_{i})-g(C_{i})\right|>\sqrt{\frac{2D_{1}\theta}{n^{2}h}}+\frac{2g_{max}\theta}{3n}+D_{2}\cdot h^{\beta}\right)\leq \exp(-\theta)
		\]
		where $D_{1}=g_{max}\int_{-\infty}^{\infty}K_m^{2}\left(v\right)dv$ and
		$D_{2}=\mathcal{L}\left|\pi\right|^{\beta-m}/m!
		\int_{-\infty}^{\infty}{\left|K_m\left(v\right)\right|}\left|v\right|^{\beta}dv$
		are constants, and for some $\pi\in[0,1]$.
	\end{lem}

The proof of the lemma is based on \citet[Propositions 1.1 and 1.2]{tsybakov_introduction_2008}
together with basic concentration inequalities; the proof can be found
in Appendix \ref{sub:Proof-of-Lemma 1}.

We now move to construct finite sample bounds for the KM-CSD learning
method when $g$ is unknown using the above lemma. We assume that
$\hat{g}$ is the kernel density estimate of $g,$ such that the conditions
of Lemma~\ref{lemma on density} hold.
\begin{thm}
	\label{theorem g unknown}Assume that (A\ref{as:A1})-(A\ref{as:A4}) hold. Assume the setting
	of Lemma~\ref{lemma on density} and that $\underset{_{z\in\mathcal{Z},c\in C}}{\inf}\hat{g}\left(c\right)\geq \kappa>0,$
	for some $\kappa>0$. 
	Then for fixed $\lambda>0,\,\theta>0,\, n\ge1,\,\varepsilon>0$, we
	have with probability not less than $1-2\exp(-\theta)$ that
	
\[
	\lambda\left\Vert f_{D,\lambda}\right\Vert _{\mathcal{H}}^{2}+R_{L,P}(f_{D,\lambda})-R_{L,P}^{*}-A_{2}(\lambda)
\leq B\sqrt{\frac{2log\left(2N(\sqrt{\frac{1}{\lambda}}B_{H},\left\Vert \cdot\right\Vert _{\infty},\epsilon)\right)+2\theta}{n}}+\frac{3c_{l}\varepsilon}{\kappa}+4c_{L}\varepsilon+2\eta
\]

	where $\eta\equiv\frac{B_{1}}{2\kappa^{2}}\left(\sqrt{\frac{2D_{1}\theta}{n^{2}h}}+\frac{2g_{max}\theta}{3n}+D_{2}\cdot h^{\beta}\right)$.
\end{thm}
The proof of the theorem appears in Appendix \ref{sub:Proof-of-Theorem 2}.

Using the above theorem we show that under some mild conditions, the KM-CSD
decision function converges in probability to the conditional expectation.
\begin{cor}
	\label{consistency 2}Assume the setting
	of Theorem~\ref{theorem g unknown} and assume that Assumption~(A\ref{as:A5}) holds. Assume that there exist constants $a\geq 1$ and $p>0$ such that $\log\left(N(B_{H},\left\Vert \cdot\right\Vert _{\infty},\epsilon)\right)
	\\ \leq a\epsilon^{-2p}$. Let $\lambda_{n}$ be a sequence such that $\lambda_{n}\underset{n\rightarrow\infty}{\rightarrow}0$
	and that $\lambda_{n}^{1+p}n\underset{n\rightarrow\infty}{\rightarrow}\infty.$
	Then the KM-CSD learning
	method is $\mathcal{P}$-universal consistent. \end{cor}
	
	The proof of this theorem appears in Appendix~\ref{sub:Proof-of-cor 2}.

We refer the readers to Appendix~\ref{Learning rates} for a straightforward derivation of learning rates that are based on the same oracle inequalities of Theorem~\ref{ theorem 1 - g known} and \ref{theorem g unknown}.
\section{\label{sec:Simulation-Study}Simulation Study}

We test the KM-CSD learning method on simulated data
and compare its performance to current state of the art. We construct
four different data-generating mechanisms, including one-dimensional
and multi-dimensional settings. For each data type, we compute the
difference between the KM-CSD decision function and the true expectation.
We compare this result to results obtained by the Cox model and by
the AFT model. We were not able to gain access to other nonparametric methods and hence will not compare them to our approach. As a reference, we compare all these methods to the
Bayes risk, which we calculated for the simulations.

For each data setting, we considered two cases: (i) the censoring
density $g$ is known; and (ii) the censoring density is unknown.
For the second setting, the distribution of the censoring variable
was estimated using univariate nonparametric kernel density estimation with a
Gaussian kernel. For simplicity, we assumed that the censoring time $C$ is independent of the covariates $Z$. The code was written in Matlab, using the Spider library\footnote{The Spider library for Matlab can be downloaded from \url{https://people.kyb.tuebingen.mpg.de/spider/main.html}.}. In order to fit the Cox model to current status data, we downloaded
the `ICsurv' R package \citep{mcmahan_icsurv:_2014}. In this package,
monotone splines are used to estimate the cumulative baseline hazard
function, and the model parameters are then chosen via the EM algorithm.
We chose the most commonly used cubic splines. To choose the number
and locations of the knots, we followed \citet{ramsay_monotone_1988}
and \citet{mcmahan_regression_2013} who both suggested using a fixed
small number of knots and thus we placed the knots evenly at the quartiles.
For the AFT model, we used the `surv reg' function in the `Survival'
R package \citep{therneau_survival:_2016}. In order to call R through
Matlab, we installed the R package rscproxy \citep{baier_rscproxy:_2012},
installed the statconnDCOM server \citep{baier_excel_2007}, and download the Matlab R-Link toolbox \citep{henson_matlab_2004}. For
the kernel of the RKHS $\mathcal{H}$, we used both a linear kernel
and a Gaussian RBF kernel $k(x_{i},x_{j})=\exp\left(\left\Vert x_{i}-x_{j}\right\Vert _{2}^{2}/2\sigma^{2}\right)$,
where $\sigma$ and $\lambda$ were chosen using 5-fold cross-validation. Cross validation is commonly used for kernel machine parameter selection (see, for example, \citealt{steinwart_support_2008}). \cite{vaart_oracle_2006} developed oracle inequalities for penalized risk minimization with multi-fold cross validation. This result can be applied to kernel machines and justifies the use of cross validation for parameter selection. Since in our case the failure time $T$ is not observed, using cross-validation with current status data is not trivial. However, one can use cross-validation with respect to our proposed loss. Recall that the risk with respect to the original loss function equals to the risk with respect to our proposed loss function. Hence we used 5-fold cross-validation with respect to the empirical risk obtained by our proposed loss.
The code for the algorithm and for the simulations can be found in the \nameref{sub:Supplementary-Material}.

We consider the following four failure time distributions, corresponding
to the four different data-generating mechanisms: (1) Weibull, (2)
Multi-Weibull, (3) Multi-Log-Normal, and (4) an additional example
where the failure time expectation is triangle shaped. We present
below the KM-CSD risks for each case and compare them to risks obtained
by other methods. The risks are based on 100 iterations per sample
size. The Bayes risk is also plotted as a reference. The Bayes risk was calculated based on the Monte Carlo method where a large number of observations were drawn from the true failure time distribution; the empirical risk was then calculated.

In Setting 1 (Weibull failure-time), the covariates $Z$ are generated
uniformly on $[0,1],$ the censoring variables $C$ is generated uniformly
on $[0,\tau],$ and the failure time $T$ is generated from a Weibull
distribution with parameters $scale=\exp(-0.5Z),\, shape=2$.
The failure time was then truncated at $\tau=1$.

Figure~\ref{fig:Weibull-failure-time} compares the results obtained
by the KM-CSD to results achieved by the Cox proportional hazards (PH) model and by the AFT
model, for different sample sizes. 

\begin{figure}[!t]
	\includegraphics[scale=0.8]{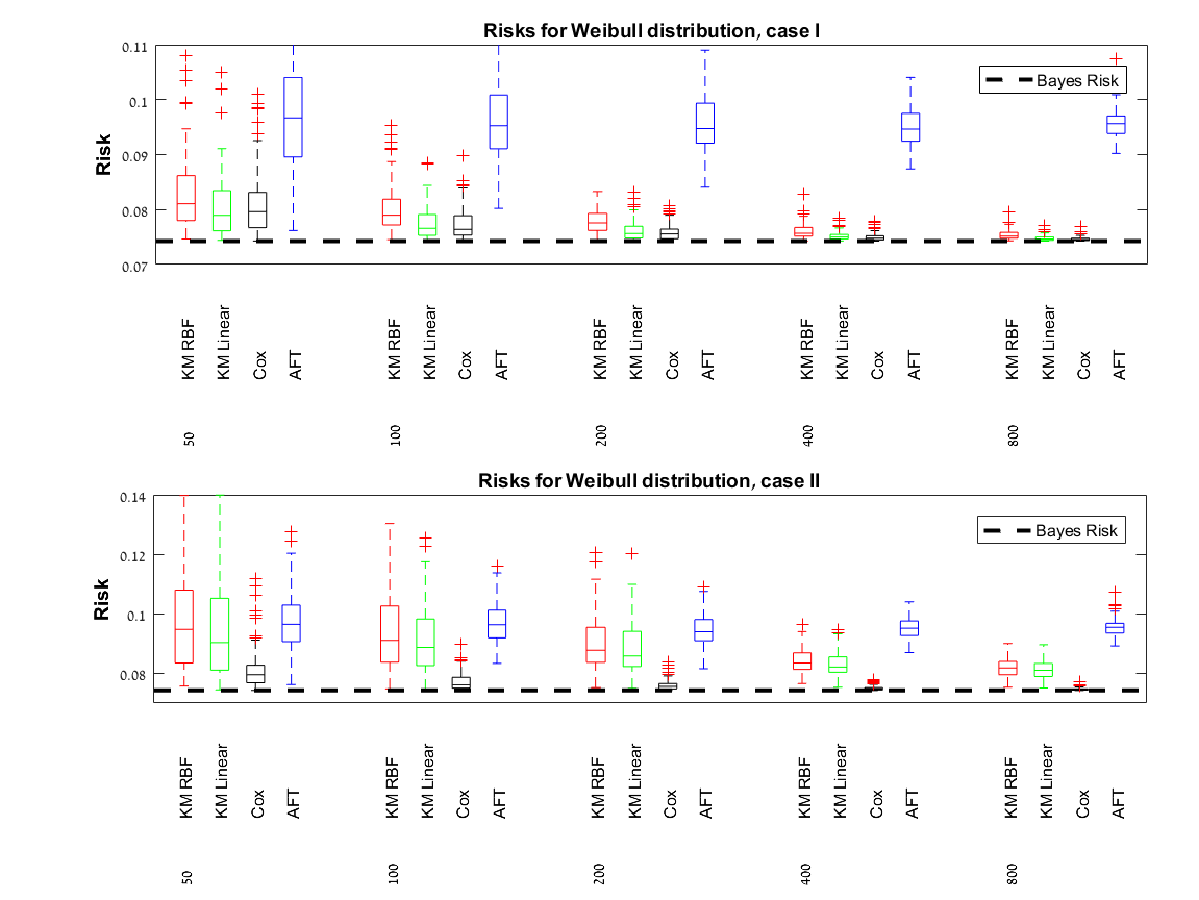}\protect\caption{Weibull failure time distribution. The Bayes risk is the dashed
		line and the boxplots of the following risks are compared: KM-CSD
		with an RBF kernel, KM-CSD with a linear kernel, Cox, and AFT, for
		sample sizes $n=50,100,200,400,800$.\label{fig:Weibull-failure-time}}
\end{figure}
In Setting 2 (Multi-Weibull failure-time), the covariates $Z$ are
generated uniformly on $[0,1]^{10},$ and the censoring variable $C$
is generated uniformly on $[0,\tau]$, as in setting 1. The failure
time $T$ is generated from a Weibull distribution with parameters
$scale=-0.5Z_{1}+2Z_{2}-Z_{3},\, shape=2$. The failure time was then
truncated at $\tau=2$. Note that this model depends only on the first
three variables. In Figure~\ref{fig:Multi-Weibull-failure-time},
boxplots of risks are presented.

\begin{figure}[!t]
	\includegraphics[scale=0.8]{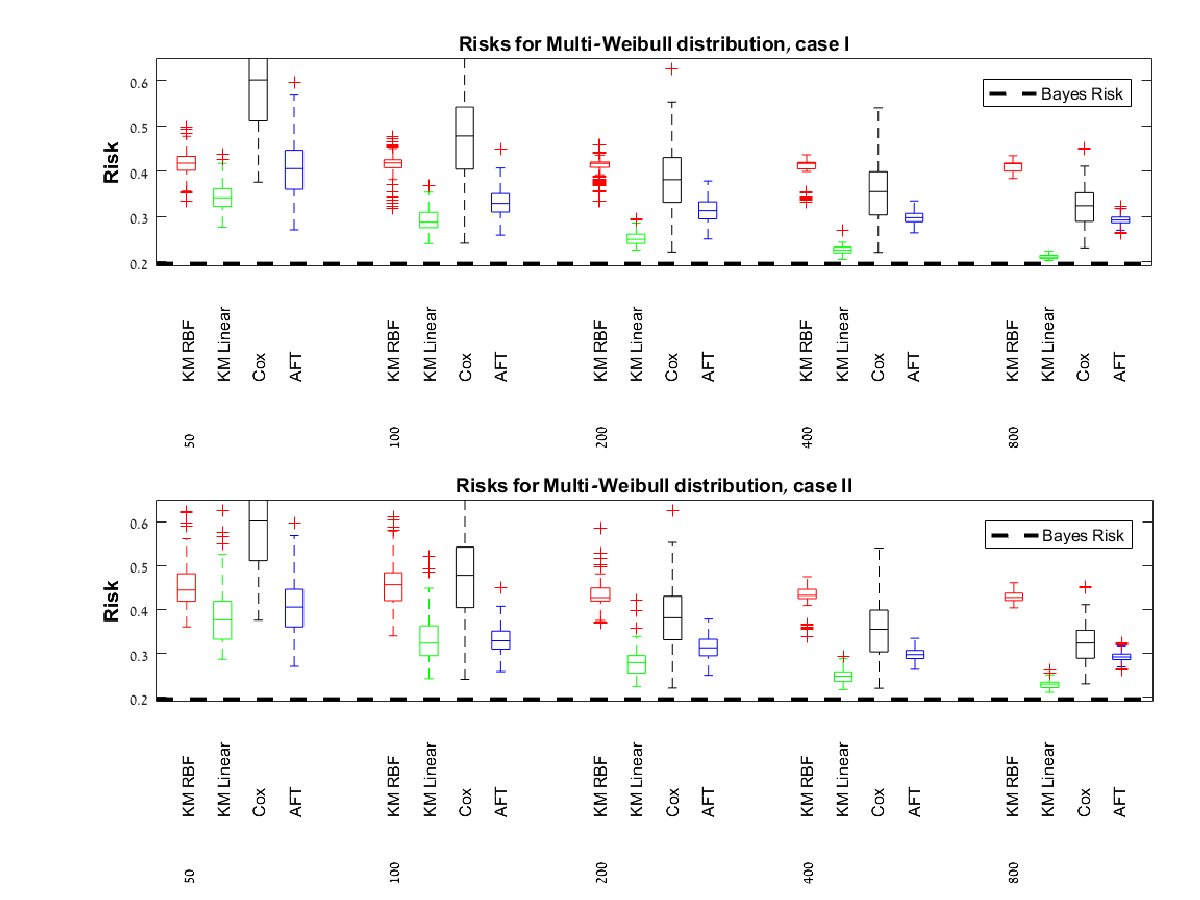}
	
	\protect\caption{Multi-Weibull failure time distribution. The Bayes risk is the dashed
		line and the boxplots of the following risks are compared: the
		KM-CSD with an RBF kernel, the KM-CSD with a linear kernel, Cox,
		and AFT, for sample sizes $n=50,100,200,400,800$. \label{fig:Multi-Weibull-failure-time}}
\end{figure}

In Setting 3 (Multi-Log-Normal), the covariates $Z$ are generated
uniformly on $[0,1]^{10},$ $C$ was generated as before and the failure
time $T$ was generated from a Log-Normal distribution with parameters
$\mu=\frac{1}{2}(0.3Z_{1}+0.5Z_{2}+0.2Z_{3}),\,\sigma=1$. The failure
time was then truncated at $\tau=7$. Figure~\ref{fig:Log-Normal-failure-time}
presents the risks of the compared methods.

\begin{figure}[!t]
	\includegraphics[scale=0.8]{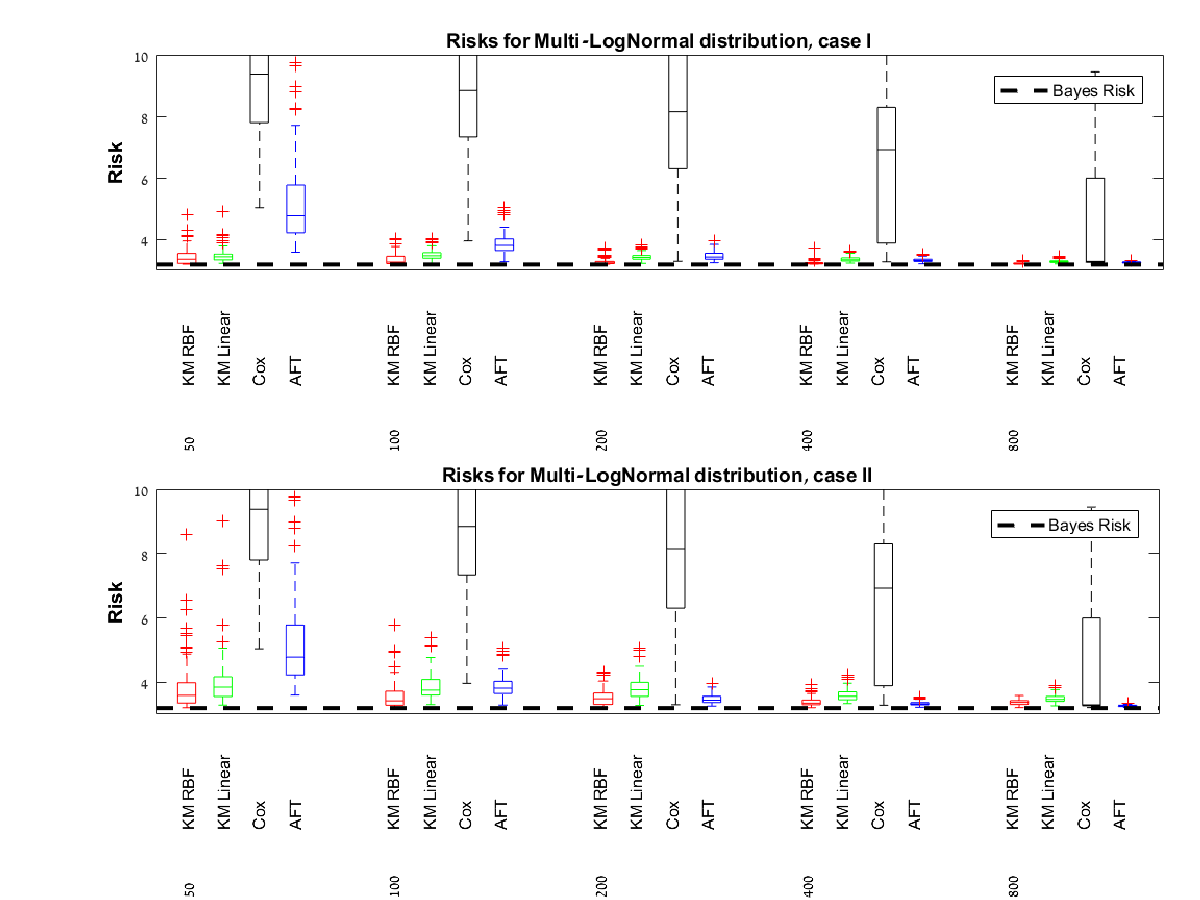}
	
	\protect\caption{Multi-LogNormal failure time distribution. The Bayes risk is the dashed
		line and the boxplots of the following risks are compared: the
		KM-CSD with an RBF kernel, the KM-CSD with a linear kernel, Cox,
		and AFT, for sample sizes $n=50,100,200,400,800$. \label{fig:Log-Normal-failure-time}}
\end{figure}

In Setting 4, we considered a non-smooth conditional expectation function
in the shape of a triangle. The covariates $Z$ are generated uniformly
on $[0,1],$ $C$ is generated uniformly on $[0,\tau]$, and $T$
is generated according to the following
\[
T=\begin{cases}
4+6\cdot Z+\epsilon & ,Z\leq0.5\\
10-6\cdot Z+\epsilon & ,Z>0.5
\end{cases},\,\, \mathrm{where}\,\,\epsilon\sim N(0,1).
\]
The failure time was then truncated at $\tau=8$. In Figure~\ref{fig:triangle box I}, the boxplots of risks are presented.

\begin{figure}[!t]
	\includegraphics[scale=0.8]{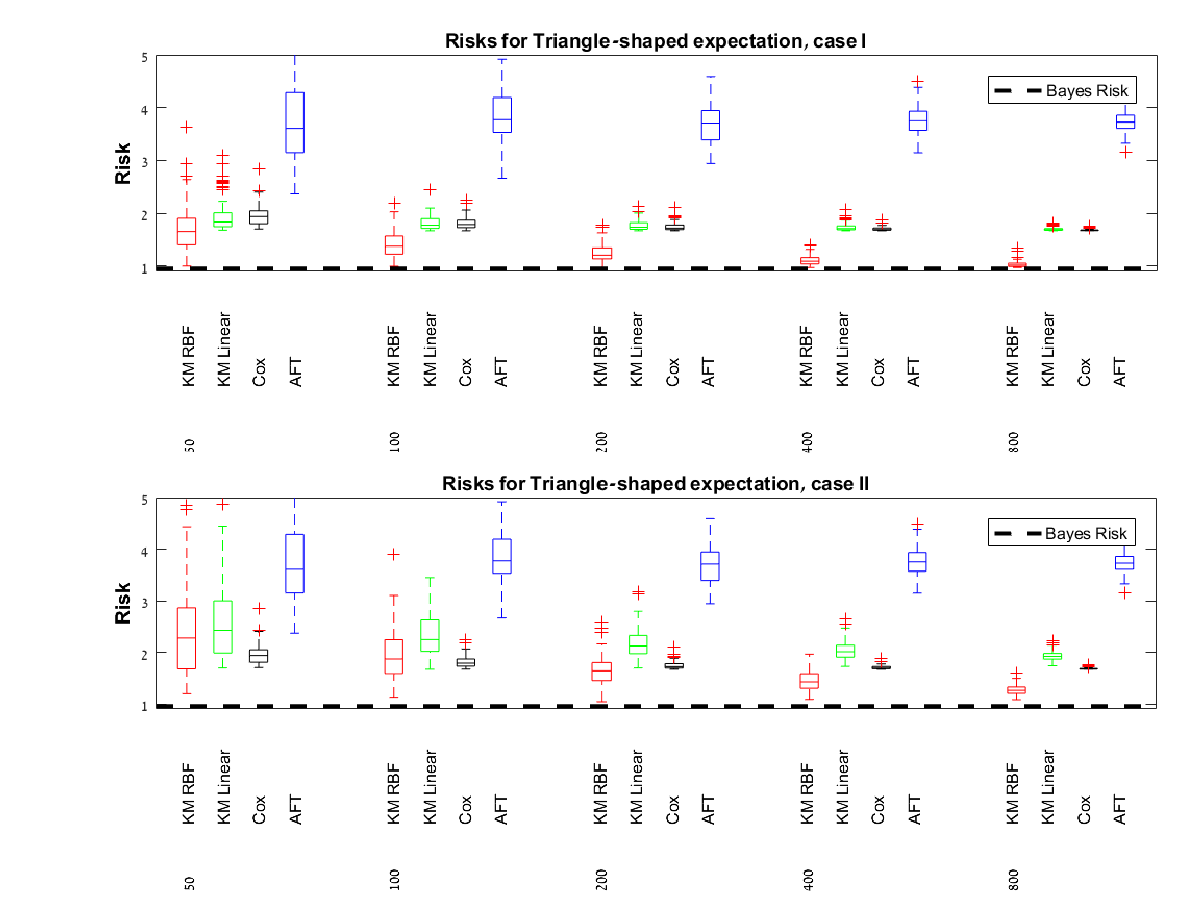}
	
	\protect\caption{Triangle shaped failure time expectation. The Bayes risk is the dashed
		line and the boxplots of the following risks are compared: the
		KM-CSD with an RBF kernel, the KM-CSD with a linear kernel, Cox,
		and AFT, for sample sizes $n=50,100,200,400,800$. \label{fig:triangle box I}}
\end{figure}

To illustrate the flexibility of the KM-CSD, we also present a graphical
representation of the true conditional expectation and its estimates,
as a function of the covariates. Figure~\ref{fig:triangle plot I}
compares the true expectation to the computed estimates for the case
that $g$ is known; these estimates are based on the first iteration.
As can be seen, the KM-CSD with an RBF kernel produces the most superior
results.

\begin{figure}[!t]
	\begin{center}
		\includegraphics[scale=0.9]{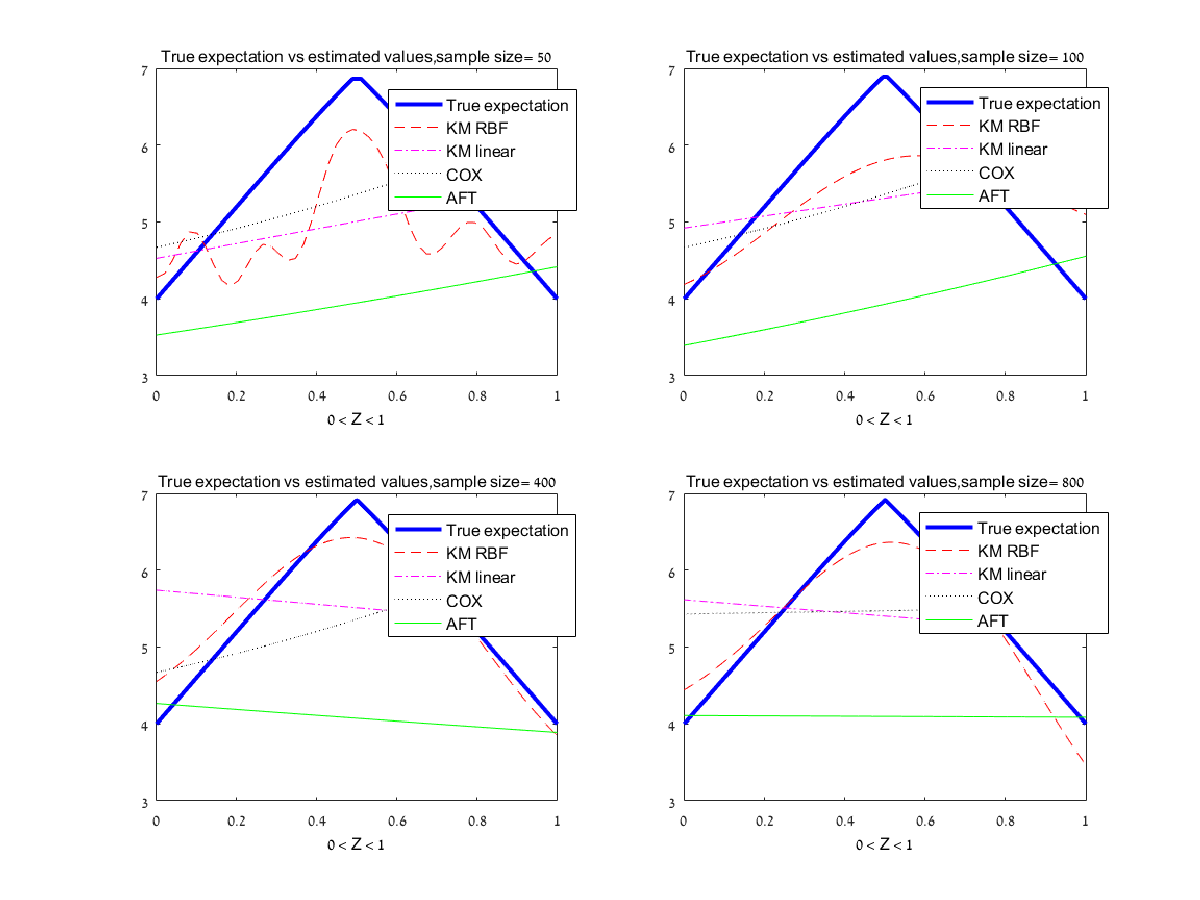}
		
		\protect\caption{Triangle shaped failure time expectation, case I ($g$ is known).
			The true expectation is the thick triangle shaped line. The following estimates are
			compared: the KM-CSD with an RBF kernel, the KM-CSD with a linear
			kernel, Cox, and AFT, for sample sizes $n=50,100,400,800$.\label{fig:triangle plot I}}
	\end{center}
\end{figure}

To summarize, Figures \ref{fig:Weibull-failure-time}-\ref{fig:triangle plot I}
showed that the KM-CSD is comparable to other known methods for estimating
the failure time distribution with current status data, and in certain
cases is even better. Specifically, we found that the KM-CSD with
an appropriate kernel was superior in three out of the four examples,
especially when the true density $g$ is known. It should be noted
that even when the assumptions of the other models were true, the KM-CSD
estimates were comparable. Additionally, when these assumptions fail
to hold, the KM-CSD estimates were generally better. The main advantage
of the proposed kernel machines approach is that it does not assume any parametric
form and thus may be superior, especially when the assumptions of
other models fail to hold. Additionally, it seems that the KM-CSD
can perform well in higher dimensions.

\section{Real World Data Analysis}\label{sec:Case Study}
In this section we test our approach on two real-world data sets, and compare its performance to current state of the art. The first data set is current status data from immunological studies, and the second is real world data concerning news popularity, with artificial censoring. Note that the second data set was artificially censored by us, allowing us to train our method on current status data, and to test it on the true uncensored data. We used the mean squared error (MSE) in order to determine the best fit. 
\subsection{Current Status Data from Immunological Studies}\label{sec:study1}

We present an analysis of real world serological data\footnote{ This dataset can be found at \url{https://www.dropbox.com/s/h120ml7pc68u63d/RCodeBook.zip?dl=0}.} on PVB19 and VZV infections. Both PVB19 and VZV cause a variety of diseases that mainly occur in childhood. The data was collected in Belgium between 2001 and 2003, as described in \cite{hens_modeling_2012}. Blood samples were tested for presence of infection-specific IgG antibodies, reflecting infection experience. In addition, age at the time of data collection was registered. These blood samples are classified as either being seropositive or seronegative, based on some cut-off level, thus yielding current status data, with patient age being the monitoring time. The statistical analysis included in this paper is based on serological data on 2382 subjects with known immunological status for both PVB19 and VZV.

For our analysis, we use the patient's age at the time of data collection as the monitoring time ($C$). We consider the continuous IgG antibody level of B19 as a covariate ($Z$) explaining the presence of the current status indicator VZV ($\Delta$). Note that we are treating the IgG antibody level of B19 as a baseline covariate, since we only have a single measurement of this antibody level. Also note that \cite{hens_modelling_2008} and \cite{abrams_modeling_2015} have investigated the association between VZV and B19, and have shown that they share the same transmission route. Hence, there is a scientific justification for using the continuous IgG antibody level of B19 as a covariate explaining the presence of VZV. 


We test our proposed KM-CSD on this data and compare it to estimates obtained by the Cox model and the AFT model. For the kernel of the RKHS $\mathcal{H}$, we used both a linear kernel and a Gaussian RBF kernel, where the kernel width $\sigma$ and the regularization parameter $\lambda$ were chosen using 5-fold cross-validation. It should be noted that we first standardized the covariates $Z$ (PVB19 antibody level) in order to suggest a reasonable selection of kernel widths. As before, the density of the censoring variable
was estimated using nonparametric kernel density estimation with a
normal kernel. In Figure \ref{fig:B19}, we present the results of the estimated expectation of time-to-infection of VZV, as a function of the covariates, for all four methods: KM with an RBF kernel, KM with a linear kernel, Cox, and AFT. It should be noted that since we do not know the true time-to-infection, we cannot argue that any model is superior. All four methods agree that there is a decreasing linear connection between time to infection of VZV, and B19 antibody level. In other words, the higher the level of PVB19, the lower the age of infection with VZV. This outcome supports previous research on joint transmission routes of VZV and B19. Further serological research can be done in order to better understand this relationship.

\begin{figure}[!t]
	\includegraphics[scale=0.26]{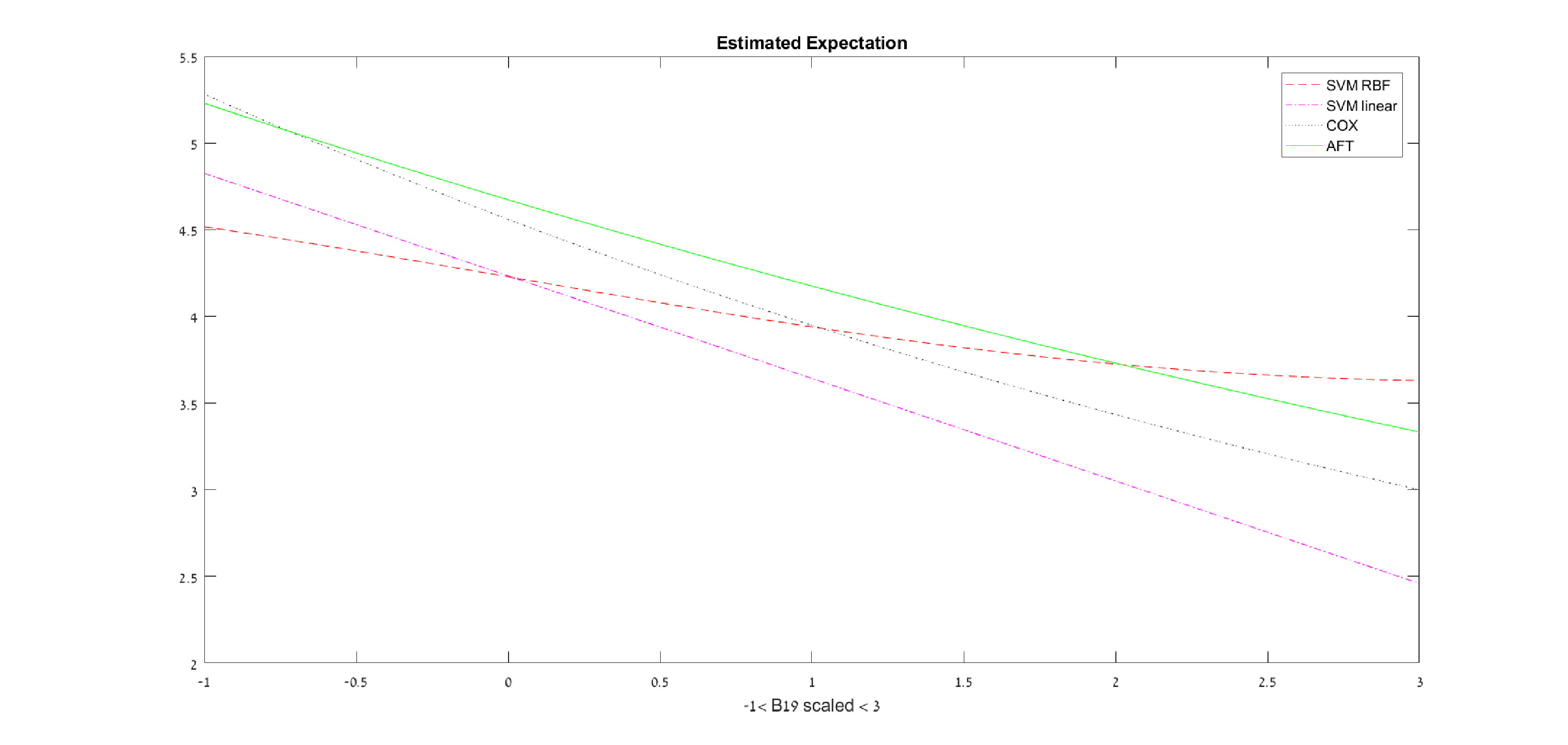}
	
	\protect\caption{Conditional expectation of time to infection of VZV as a function of the standardized antibody level of B19. The following estimates are
		compared: the KM-CSD with an RBF kernel, the KM-CSD with a linear
		kernel, Cox and AFT. \label{fig:B19}}
\end{figure}

\subsection{Artificially Censored Real-World Data}\label{sec:study2}
For our second analysis, we used real-world data on news popularity\footnote{This dataset can be found at  \url{https://archive.ics.uci.edu/ml/datasets/Online+News+Popularity}.}, with artificial censoring. The original data summarizes a set of features regarding articles published by Mashable, in a period of two years, as described in \cite{fernandes_proactive_2015}.  The goal is to predict the number of shares of an article in social networks, referred to as `popularity'. Since the number of shares is non-negative, we consider it as our failure-time $T$. The original dataset contains 58 predictive attributes. As before, we first standardized the covariates $Z$. In order to reduce the dimensionality of the data, we used the LASSO method for subset selection \citep{tibshirani_regression_1996}. For the sake of our analysis, we used the six most important explanatory variables. In order to obtain current status data, we generated the monitoring times $C_1,...,C_n$ as random exponential variables with mean equal to the mean number of shares. We then calculated the current status indicator by $\Delta=\boldsymbol{1}_{\{T<C\}}$. In summary, the artificially censored data consists of six covariates, the current status indicator, and the monitoring time generated from an exponential distribution. The uncensored data after standardization and dimensionality reduction, and its artificially censored version, can be found in the \nameref{sub:Supplementary-Material}.

Since the original dataset contains 39,644 entries, we divided it randomly into 35 training sets of 1000 observations, and one testing set of 4644 observations. The training sets consisted of the artificially censored data, whereas the testing data contained the original uncensored scaled number of shares. We trained the KM-CSD, with both a linear and a Gaussian RBF kernel, as well as Cox and AFT, on each training set. As before, the kernel width $\sigma$ and the regularization parameter $\lambda$ were chosen using 5-fold cross-validation. For a fair comparison, we estimated the density of the censoring variable using nonparametric kernel density estimation with a
normal kernel, and did not use our knowledge regarding the censoring mechanism. For each training set, we computed the model predictions on the testing set and calculated the corresponding MSE. Since the MSE is sensitive to the overall
scale of the response variable, we divided the MSE by the empirical variance of the number of shares in order to achieve standardized MSE. Figure~\ref{fig:MSEs} presents the boxplot of the standardized MSEs (SMSEs), for all four methods: KM-CSD with an RBF kernel, KM-CSD with a linear kernel, Cox, and AFT. Figure~\ref{fig:MSEs} shows that the SMSEs produced by the KM-CSD, with either a linear or a Gaussian kernel, is similar to the SMSEs produced by the Cox model, and is significantly lower than the SMSEs produced by the AFT model. In fact, the KM-CSD with a linear kernel produced the lowest SMSEs, whereas the KM-CSD with a Gaussian RBF kernel produced the SMSEs with the lowest variance. Additionally, for better readability of the results, we split the results into two sub-figures since the AFT produced much larger SMSEs than the other methods. It should also be noted that for some training sets, the AFT SMSE was so high that we had to omit it from the graphical representation. 
\begin{figure}[!t]

	\includegraphics[scale=0.26]{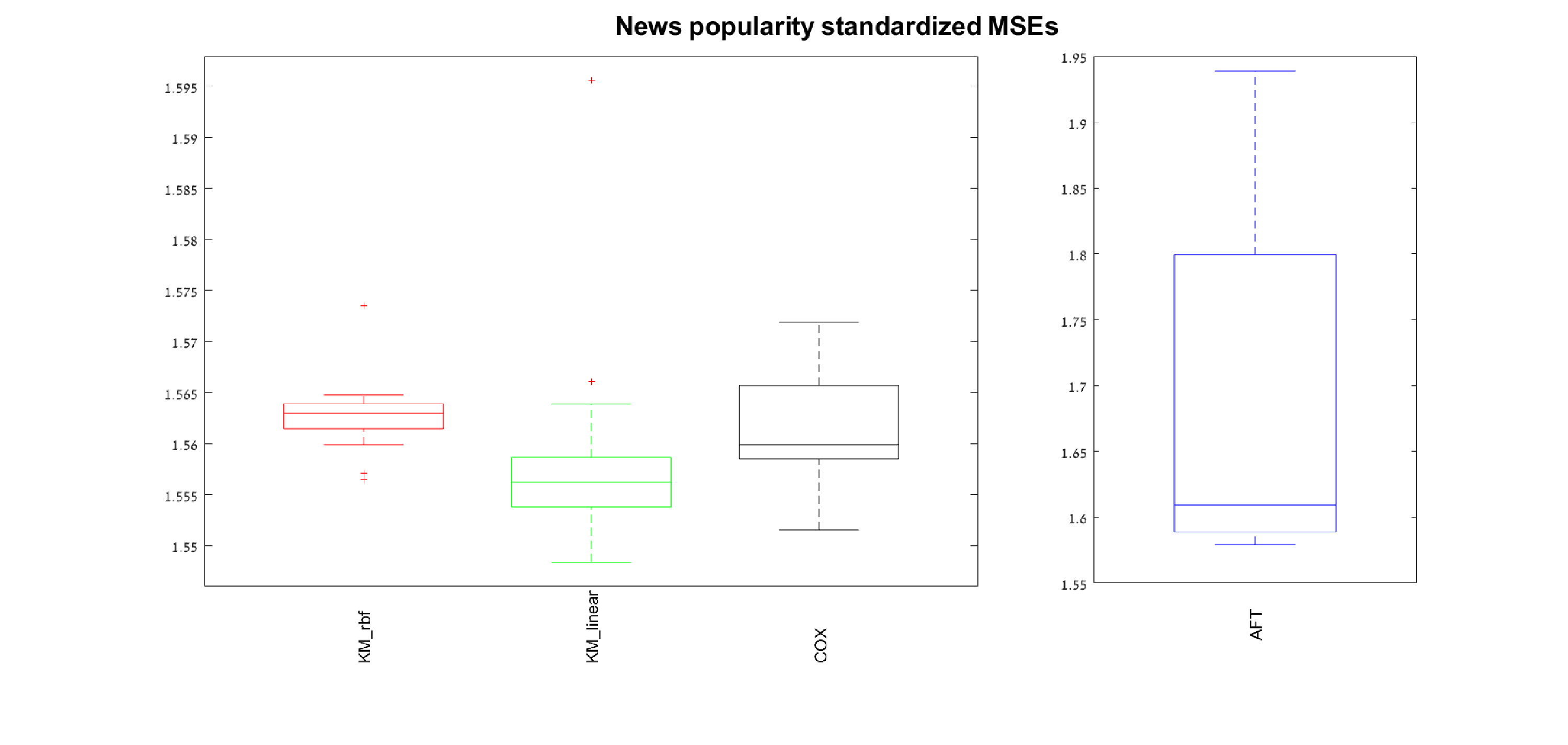}
	
	\protect\caption{SMSEs of predicted number of shares, based on 35 training sets. The following estimates are
		compared: the KM-CSD with an RBF kernel, the KM-CSD with a linear
		kernel, Cox, and AFT. Note the different scale between the left and right subplots.  \label{fig:MSEs}}

\end{figure}

\section{\label{sec:Concluding-Remarks}Concluding Remarks}

We proposed a kernel-machine approach for estimation of the failure time expectation,
studied its theoretical properties, presented a simulation study, and tested our approach on two real-world data sets. Specifically, we proved that our method is consistent, and showed by simulations and analysis of real-world data that our approach is just as good as current state of the art, and sometimes even better.
We believe this work demonstrates an important approach in applying
machine learning techniques  to current status data. However, many
open questions remain and many possible generalizations exist. First,
note that we only studied the problem of estimating the failure time
expectation and not other distribution related quantities. Further
work needs to be done in order to extend the kernel machines approach to other
estimation problems with current status data, and is beyond the scope of this paper. Note that the theory developed here might not hold in such generalizations, as the corresponding modified loss function will no longer be a convex function. Second, we assumed that
the censoring is independent of the failure time given the covariates
and that the censoring density is positive given the covariates over
the entire observed time range. It would be worthwhile to study the
consequences of violation of some of these assumptions. Third, it
could be interesting to extend this work to other censored data formats
such as interval censoring. We believe that further development and
generalization of kernel machine learning methods for different types of censored
data is of great interest. Some additional generalization of this work can include derivation of doubly-robust estimators and inclusion of time-dependent covariates. 
For the case of time-dependent covariates, one first needs to define an RKHS over the covariate process space and then to define the appropriate empirical risk minimization. Since this space is rich, the covering number results discussed in Section \ref{sec:Theoretical-Results} may not hold for this space. 

\section*{\label{sub:Supplementary-Material}Supplementary Materials}
The appendices referenced in Sections~\ref{sec:Kernel-Machines} and \ref{sec:Theoretical-Results}, the Matlab code for the algorithm, simulations, and data-analysis, and the artificially censored data used in Section \ref{sec:study2}, are available with this article.
		\subsection*{Matlab code} Folder `SVR for CSD' containing the Matlab code for both the algorithm and for
		the simulations. Please read the README.pdf for details on the files in this folder. A link to the folder can be \href{https://www.dropbox.com/sh/v4xbx5mv3im9xns/AAChkLY6A4lcg5m_8JjPMHdva?dl=0}{found here}.

		\subsection*{Artificially censored data set} The \href{https://www.dropbox.com/sh/zl7ubx0to69g868/AAAvd-aWFJV5TJo5WhMGPW8ba?dl=0}{data and the code} for the data analysis in Section \ref{sec:study2}. This includes the uncensored data after standardization and dimensionality reduction and its artificially censored version, and the code that performs the analysis and produces Figure~\ref{fig:MSEs}. Both data sets are Matlab .mat files, and the data analysis code is a Matlab .m file. The code is based on the functions defined in the folder `SVR for CSD' described above.

%
%
%

\section*{Acknowledgements}
The authors would like to thank Niel Hens for sharing the serological dataset on VZV and PVB19, and for helpful discussions.

\section*{Conflict of interest}
The authors declare that they have no conflict of interest.
\appendix
	\section*{Appendix}
    \renewcommand{\thesubsection}{\Alph{subsection}}
	
	\subsection{\label{closed-form} Computation of the Decision Function}
	
	Equation~\ref{eq: KM-CSD} is a quadratic optimization problem. Such problems are vastly studied in the literature (see, for example, \citealt{suykens_least_1999}) and known solutions exist. Specifically, using the representer theorem (\citealt[Theorem 5.5]{steinwart_support_2008}), the solution of Equation~\ref{eq: KM-CSD} is given by \[f_{D,\lambda}(Z)=\sum_{i=1}^n \alpha_{i}k(Z,Z_{i})+b.\] Using the Lagrange method, the quadratic optimization problem in Eq.~\ref{eq: KM-CSD} can be simplified to a set of linear equations (see, for example, \citealt{fletcher_practical_1987}). Hence, it can be shown that the coefficients  $\alpha_{1},...,\alpha_{n}$ and $b$ in the representation of $f_{D,\lambda}$ above can be obtained by 
	\[
	\left(\begin{array}{c}
	\alpha_{1}\\
	\alpha_{2}\\
	.\\
	.\\
	.\\
	\alpha_{n}\\
	b
	\end{array}\right)=\left(\begin{array}{ccccccc}
	K_{11}+n\lambda & K_{12} & . & . & . & K_{1n} & 1\\
	K_{21} & K_{22}+n\lambda & . & . & . & K_{2n} & 1\\
	. & . & . &  &  & . & .\\
	. & . &  & . &  & . & .\\
	. & . &  &  & . & . & .\\
	K_{n1} & K_{n2} & . & . & . & K_{nn}+n\lambda & 1\\
	1 & 1 & . & . & . & 1 & 0
	\end{array}\right)^{-1}.\left(\begin{array}{c}
	v_{1}\\
	v_{2}\\
	.\\
	.\\
	.\\
	v_{n}\\
	0
	\end{array}\right),
	\]
	where $\mathbf{K}_{n\times n}$ is the kernel matrix with entries $\mathbf{K}_{ij}=k(Z_{i},Z_{j})$, and where $v_i=(1-\Delta_{i})/g(C_{i}|Z_{i})$, for $1\leq i,j \leq n$. That is, the KM-CSD decision function has a closed form. 
	
	\subsection{\label{positivity}Non-negative new modified loss function }
	Recall that our proposed loss function is
	\[
	L^{n}(D,(Z,C,\Delta,s))=\frac{(1-\Delta)2(C-s)}{g(C|Z)}+s^{2}\,.
	\]
	Note that this function is convex but not necessarily a loss function
	since it can obtain negative values.
	In order  to ensure positivity
	we add a constant term that does not depend on $f$, and so our loss
	becomes \[\widetilde{L^{n}}(D,(Z,C,\Delta,f(Z)))=\frac{(1-\Delta)2(C-f(Z))}{g(C|Z)}+\left(f(Z)\right)^{2}+a,
	\]
	where for a fixed dataset of length $n,$ the constant $a$ is $a=\underset{1\leq i\leq n}{\max}\left\{ (1-\Delta_{i})/\left(g(C_{i}|Z_{i})\right)^{2}\right\}.$
	Note that this additional term will not effect the optimization (since
	$\widetilde{L^{n}}$ is just a shift by a constant of $L^{n}$) and
	thus will be neglected hereafter.

	\subsection{\label{sec:Appendix}Proofs}
	
	\subsubsection{\label{sub:Proof-of-Theorem 1}Proof of Theorem~\ref{ theorem 1 - g known}}
	\begin{proof}
		Since $L^{n}(D,(Z,C,\Delta,s))=\tau^{-2}\left((1-\Delta)2(C-s)/g(C|Z)+s^{2}\right)$
		is convex, it implies that there exists a unique decision function \citep[see][Section 5.1]{steinwart_support_2008}.
		For all distributions $Q$ on $\mathcal{Z\times Y}$, we define the
		kernel machine decision function by $f_{Q,\lambda}=\underset{f\in\mathcal{H}}{\inf}\lambda\left\Vert f\right\Vert _{\mathcal{H}}^{2}+R_{L,Q}(f).$
		We note that for an RKHS $\mathcal{H}$ of a continuous kernel $k$
		with $\left\Vert k\right\Vert _{\infty}\le1$,
		
		\[
		\left\Vert f_{Q,\lambda}\right\Vert _{\mathcal{\infty}}\leq\left\Vert k\right\Vert _{\infty}\left\Vert f_{Q,\lambda}\right\Vert _{\mathcal{H}}\leq\left\Vert f_{Q,\lambda}\right\Vert _{\mathcal{H}}.
		\]
		Hence,
		
		\[
		\lambda\left\Vert f_{Q,\lambda}\right\Vert _{\mathcal{H}}^{2}\leq\lambda\left\Vert f_{Q,\lambda}\right\Vert _{\mathcal{H}}^{2}+R_{L,Q}(f_{Q,\lambda})=\underset{f\in\mathcal{H}}{\inf}\lambda\left\Vert f\right\Vert _{\mathcal{H}}^{2}+R_{L,Q}(f)\leq\lambda\left\Vert 0\right\Vert _{\mathcal{H}}^{2}+R_{L,Q}(0)=R_{L,Q}(0),
		\]

		Hence $\left\Vert f_{Q,\lambda}\right\Vert _{\mathcal{\infty}}\leq\left\Vert f_{Q,\lambda}\right\Vert _{\mathcal{H}}\leq\lambda^{-\nicefrac{1}{2}}\sqrt{R_{L,Q}(0)}$
		for all $f\in\mathcal{H}$. By Remark \ref{rem:loss bounds}, $L(y,0)\leq1$
		for all $y\in\mathcal{Y}$ and so we conclude that $R_{L,Q}(0)\leq1$
		and thus $\left\Vert f_{Q,\lambda}\right\Vert _{\mathcal{\infty}}\leq\left\Vert f_{Q,\lambda}\right\Vert _{\mathcal{H}}\leq \lambda^{-\nicefrac{1}{2}}$
		for all distributions $Q$ on $\mathcal{Z\times Y}$.
		
		Recall that the unit ball of $\mathcal{H}$ is denoted by $B_{H}$
		and its closure by $\overline{B_{H}}$; since$\left\Vert f_{P,\lambda}\right\Vert _{\mathcal{H}}\leq\lambda^{-\nicefrac{1}{2}}$
		we can write $f\in\lambda^{-\nicefrac{1}{2}}\overline{B_{H}}$. Since
		$\mathcal{Z}\subset\mathbb{\mathbb{R}}^{d}$ is compact, it implies
		that the $\left\Vert \cdot\right\Vert _{\infty}-\mathrm{closure}$
		$\overline{B_{H}}$ of the unit ball $B_{H}$ is compact in $\ell_{\infty}(\mathcal{Z})$
		\citep[see][Corollary 4.31]{steinwart_support_2008}.
		
		Denote by $R_{L^n,D}(f)$ the empirical risk with respect to the data-dependent loss $L^n$. Since $f_{D,\lambda}$ minimizes $\lambda\left\Vert f\right\Vert _{\mathcal{H}}^{2}+R_{L^n,D}(f)$,
		\[
		\lambda\left\Vert f_{D,\lambda}\right\Vert _{\mathcal{H}}^{2}+R_{L^n,D}(f_{D,\lambda})\leq\lambda\left\Vert f_{P,\lambda}\right\Vert _{\mathcal{H}}^{2}+R_{L^n,D}(f_{P,\lambda}).
		\]
		Recall that the approximation error is defined by $A_{2}(\lambda)=\underset{f\in\mathcal{H}}{\inf}\lambda\left\Vert f\right\Vert _{\mathcal{H}}^{2}+R_{L,P}(f)-R_{L,P}^{*}$,
		and thus, as in \citet[Eq. 6.18]{steinwart_support_2008},
		
		\begin{alignat*}{1}
		& \lambda\left\Vert f_{D,\lambda}\right\Vert _{\mathcal{H}}^{2}+R_{L,P}(f_{D,\lambda})-R_{L,P}^{*}-A_{2}(\lambda)\\
		= & \lambda\left\Vert f_{D,\lambda}\right\Vert _{\mathcal{H}}^{2}+R_{L,P}(f_{D,\lambda})-\lambda\left\Vert f_{P,\lambda}\right\Vert _{\mathcal{H}}^{2}-R_{L,P}(f_{P,\lambda})\\
		= & \lambda\left\Vert f_{D,\lambda}\right\Vert _{\mathcal{H}}^{2}+R_{L^n,D}(f_{D,\lambda})-R_{L^n,D}(f_{D,\lambda})+R_{L,P}(f_{D,\lambda})-\lambda\left\Vert f_{P,\lambda}\right\Vert _{\mathcal{H}}^{2}-R_{L,P}(f_{P,\lambda})\\
		\leq & \lambda\left\Vert f_{P,\lambda}\right\Vert _{\mathcal{H}}^{2}+R_{L^n,D}(f_{P,\lambda})-R_{L^n,D}(f_{D,\lambda})+R_{L,P}(f_{D,\lambda})-\lambda\left\Vert f_{P,\lambda}\right\Vert _{\mathcal{H}}^{2}-R_{L,P}(f_{P,\lambda})\\
		= & R_{L^n,D}(f_{P,\lambda})-R_{L^n,D}(f_{D,\lambda})+R_{L,P}(f_{D,\lambda})-R_{L,P}(f_{P,\lambda})\\
		\leq & 2\underset{\left\Vert f\right\Vert _{\mathcal{H}}\leq\lambda^{-\nicefrac{1}{2}}}{\sup}|R_{L,P}(f)-R_{L^n,D}(f)|.
		\end{alignat*}

		That is,
		
		\begin{equation}
		\lambda\left\Vert f_{D,\lambda}\right\Vert _{\mathcal{H}}^{2}+R_{L,P}(f_{D,\lambda})-R_{L,P}^{*}-A_{2}(\lambda)\leq2\underset{\left\Vert f\right\Vert _{\mathcal{H}}\leq\lambda^{-\nicefrac{1}{2}}}{\mathcal{\sup}}|R_{L,P}(f)-R_{L^n,D}(f)|\label{eq:first bound}
		\end{equation}

		Note that since $L$ is Lipschitz continuous, $|L(y,s)-L(y,s')|\leq c_{L}|s-s'|$
		for all $s,s'\in S$.
		
		From the discussion above, we are only interested in bounded functions
		$f\in\lambda^{-\nicefrac{1}{2}}\overline{B_{H}}.$
		
		Then for all $f\in\lambda^{-\nicefrac{1}{2}}\overline{B_{H}}$ we have
		
		\[
		|L(y,f(z))|\text{\ensuremath{\le}}|L(y,f(z))-L(y,0)|+L(y,0)\leq c_{L}|f(z)|+1\leq c_{L}\lambda^{-\nicefrac{1}{2}}+1\equiv B_{2}
		\]
		thus we obtain that for functions $f\in\lambda^{-\nicefrac{1}{2}}\overline{B_{H}}$,
		the loss $L(y,f(z))$ is bounded.
		
		For any $\epsilon>0,$ let $\mathcal{F_{\epsilon}}$ be an $\varepsilon-net$
		of $\lambda^{-\nicefrac{1}{2}}\overline{B_{H}}$. Since $\overline{B_{H}}$
		is compact, then the cardinality of the $\varepsilon-net$ is
		\[
		|\mathcal{F}\varepsilon|=N\left(\lambda^{-\nicefrac{1}{2}}B_{H},\left\Vert \cdot\right\Vert _{\infty},\epsilon\right)=N(B_{H},\left\Vert \cdot\right\Vert _{\infty},\sqrt{\lambda}\epsilon)<\infty.
		\]

		Thus for every $f\in\lambda^{-\nicefrac{1}{2}}\overline{B_{H}}$, there
		exists a function $h\in\mathcal{F_{\varepsilon}}$ with $\left\Vert f-h\right\Vert \le\varepsilon$,
		and thus

		\begin{equation}
		\begin{alignedat}{5}
		&|R_{L,P}(f)-R_{L^n,D}(f)|\\
		&\leq|R_{L,P}(f)-R_{L,P}(h)|+|R_{L,P}(h)-R_{L^n,D}(h)|+|R_{L^n,D}(h)-R_{L^n,D}(f)|\\
		&\equiv A_{n}+B_{n}+C_{n}\label{eq:abc}
		\end{alignedat}
		\end{equation}

		First we will bound $C_{n}$;
		
		\begin{alignat*}{1}
		C_{n}\equiv & \left|R_{L^n,D}(h)-R_{L^n,D}(f)\right|\\
		\leq & \left|\frac{1}{n}\sum_{i=1}^n \left[\frac{(1-\Delta_{i})\ell(C_{i},h(Z_{i}))}{g(C_{i}|Z_{i})}\right]-\frac{1}{n}\sum_{i=1}^n \left[\frac{(1-\Delta_{i})\ell(C_{i},f(Z_{i}))}{g(C_{i}|Z_{i})}\right]\right|\\
		+ & \left|\frac{1}{n}\sum_{i=1}^n [L(0,h(Z_{i}))]-\frac{1}{n}\sum_{i=1}^n [L(0,f(Z_{i}))]\right|\\
		\equiv & C_{n,1}+C_{n,2},
		\end{alignat*}

		where
		\begin{alignat*}{1}
		C_{n,1}\equiv & \left|\frac{1}{n}\sum_{i=1}^n \left[\frac{(1-\Delta_{i})\ell(C_{i},h(Z_{i}))}{g(C_{i}|Z_{i})}-\frac{(1-\Delta_{i})\ell(C_{i},f(Z_{i}))}{g(C_{i}|Z_{i})}\right]\right|\\
		= & \left|\frac{1}{n}\sum_{i=1}^n \left[\frac{(1-\Delta_{i})}{g(C_{i}|Z_{i})}\left(\ell(C_{i},h(Z_{i}))-\ell(C_{i},f(Z_{i}))\right)\right]\right|\\
		\leq & \left|\frac{1}{n}\sum_{i=1}^n \left[\frac{1}{g(C_{i}|Z_{i})}\left(\ell(C_{i},h(Z_{i}))-\ell(C_{i},f(Z_{i}))\right)\right]\right|\\
		\leq & \frac{1}{2\kappa}\left|\frac{1}{n}\sum_{i=1}^n \left[\ell(C_{i},h(Z_{i}))-\ell(C_{i},f(Z_{i}))\right]\right|\leq\frac{1}{2n\kappa}\sum_{i=1}^n |\ell(C_{i},h(Z_{i}))-\ell(C_{i},f(Z_{i}))|\\
		\leq & \frac{1}{2n\kappa}\sum_{i=1}^n c_{l}|h(Z_{i})-f(Z_{i})|\leq\frac{1}{2n\kappa}\sum_{i=1}^n c_{l}\varepsilon=\frac{c_{l}\varepsilon}{2\kappa},
		\end{alignat*}

		and where
		\begin{alignat*}{1}
		C_{n,2}\equiv & \left|\frac{1}{n}\sum_{i=1}^n [L(0,h(Z_{i}))-L(0,f(Z_{i}))]\right|\\
		\leq & \frac{1}{n}\sum_{i=1}^n \left|L(0,h(Z_{i}))-L(0,f(Z_{i}))\right|\\
		\leq & \frac{1}{n}\sum_{i=1}^n c_{L}|h(Z_{i})-f(Z_{i})|\leq\frac{1}{n}\sum_{i=1}^n \left[c_{L}\varepsilon\right]=c_{L}\varepsilon
		\end{alignat*}

		So we were able to bound $C_{n}$ by $c_{l}\varepsilon/2\kappa+c_{L}\varepsilon$.
		
		Similarly, using to the property that $E\left[\alpha\right]=\alpha$
		for any constant $\alpha$, it can be shown that $A_{n}\leq c_{l}\varepsilon/2\kappa+c_{L}\varepsilon$.
		
		As an interim summary, we showed that
		
		\begin{equation}
		\underset{f\in\lambda^{-\nicefrac{1}{2}}B_{H}}{\sup}|R_{L,P}(f)-R_{L^n,D}(f)|\leq\underset{h\in\mathcal{F_{\varepsilon}}}{\sup}\underset{=B_{n}}{\underbrace{|R_{L,P}(h)-R_{L^n,D}(h)|}}+\frac{1}{\kappa}c_{l}\varepsilon+2c_{L}\varepsilon.\label{eq:second bound}
		\end{equation}

		Recall that the loss $L(y,f(z))$ is bounded by $B_{2}$ and that
		by Remark \ref{rem:loss bounds}, $\ell(y,s)\leq B_{1}$.
		
		We note that
		
		\begin{alignat*}{2}
		\frac{(1-\Delta)\ell(C,h(Z))}{g(C|Z)}+L(0,h(Z)) & \leq\frac{\ell(C,h(Z))}{g(C|Z)}+L(0,h(Z))\leq & \frac{B_{1}}{2\kappa}+B_{2}\equiv B
		\end{alignat*}

		Combining this with equation (\ref{eq:first bound}), we obtain that
		
		\begin{alignat*}{2}
		Pr & \left(\lambda\left\Vert f_{D,\lambda}\right\Vert _{\mathcal{H}}^{2}+R_{L,P}(f_{D,\lambda})-R_{L,P}^{*}-A_{2}(\lambda)\geq B\sqrt{\frac{2\eta}{n}}+\frac{2c_{l}\varepsilon}{\kappa}+4c_{L}\varepsilon\right)\\
		\leq & Pr\left(2\underset{\left\Vert f\right\Vert _{\mathcal{H}}\leq\lambda^{-\nicefrac{1}{2}}}{sup}|R_{L,P}(f)-R_{L^n,D}(f)|\geq B\sqrt{\frac{2\eta}{n}}+\frac{2c_{l}\varepsilon}{\kappa}+4c_{L}\varepsilon\right) & \,(\mathrm{by\, eq.\,}\ref{eq:first bound})\\
		\leq & Pr\left(2\left(\underset{h\in\mathcal{F_{\varepsilon}}}{sup}|R_{L,P}(h)-R_{L^n,D}(h)|+\frac{1}{\kappa}c_{l}\varepsilon+2c_{L}\varepsilon\right)\geq B\sqrt{\frac{2\eta}{n}}+\frac{2c_{l}\varepsilon}{\kappa}+4c_{L}\varepsilon\right) & \,(\mathrm{by\, eq.\,}\ref{eq:second bound})\\
		= & Pr\left(2\left(\underset{h\in\mathcal{F_{\varepsilon}}}{\sup}B_{n}+\frac{1}{\kappa}c_{l}\varepsilon+2c_{L}\varepsilon\right)\geq B\sqrt{\frac{2\eta}{n}}+\frac{2c_{l}\varepsilon}{\kappa}+4c_{L}\varepsilon\right) & \,\\
		= & Pr\left(\underset{h\in\mathcal{F_{\varepsilon}}}{\sup}B_{n}\geq B\sqrt{\frac{\eta}{2n}}\right)=Pr\left(\underset{h\in\mathcal{F_{\varepsilon}}}{\sup}\left|R_{L,P}(h)-R_{L^n,D}(h)\right|\geq B\sqrt{\frac{\eta}{2n}}\right).
		\end{alignat*}

		By the union bound, the last expression is bounded by
		\begin{alignat*}{1}
		& \underset{h\in\mathcal{F_{\varepsilon}}}{\sum}Pr\left(|R_{L,P}(h)-R_{L^n,D}(h)|\geq B\sqrt{\frac{\eta}{2n}}\right),
		\end{alignat*}

		which can then be bounded again by $2|\mathcal{F}_{\varepsilon}|\exp(-\eta)$,
		using Hoeffding's inequality \citep[Theorem 6.10]{steinwart_support_2008};
		where $\mathcal{F_{\epsilon}}$ is an $\varepsilon$-net of $\lambda^{-\nicefrac{1}{2}}\overline{B_{H}}$
		with cardinality
		\[
		|\mathcal{F}\varepsilon|=N\left(\lambda^{-\nicefrac{1}{2}}B_{H},\left\Vert \cdot\right\Vert _{\infty},\epsilon\right)<\infty.
		\]

		Define $\eta=\log(2|\mathcal{F_{\varepsilon}}|)+\theta$, then
		\begin{align*}
		& Pr\left(\lambda\left\Vert f_{D,\lambda}\right\Vert _{\mathcal{H}}^{2}+R_{L,P}(f_{D,\lambda})-R_{L,P}^{*}-A_{2}(\lambda)\geq B\sqrt{\frac{2(\log(2|\mathcal{F_{\varepsilon}}|)+\theta)}{n}}+\frac{2c_{l}\varepsilon}{\kappa}+4c_{L}\varepsilon\right)\\
		&  \leq \exp(-\theta),
		\end{align*}

		which concludes the proof.
	\end{proof}
	
	\subsubsection{\label{sub:Proof-of-cor 1}Proof of Corollary~\ref{cor - consistency}}
	\begin{proof}
	\label{proof:cor - consistency} 
	In Theorem~\ref{ theorem 1 - g known} we showed that 
	\[
	\lambda\left\Vert f_{D,\lambda}\right\Vert _{\mathcal{H}}^{2}+R_{L,P}(f_{D,\lambda})-R_{L,P}^{*}-A_{2}(\lambda)\leq B\sqrt{\frac{2\log\left(2N(\sqrt{\frac{1}{\lambda}}B_{H},\left\Vert \cdot\right\Vert _{\infty},\epsilon)\right)+2\theta}{n}}+\frac{2c_{l}\varepsilon}{\kappa}+4c_{L}\varepsilon,
	\]

	with probability not less than $1-\exp^{-\theta}$.
	
	Choose $\lambda=\lambda_{n};$ from Assumption~(A\ref{as:A5}) together with
	Lemma~5.15 of \citet[5.15]{steinwart_support_2008}, $A_{2}(\lambda_{n})$
	converges to zero as $n$ goes to infinity. 
	By the assumption $\log\left(N(B_{H},\left\Vert \cdot\right\Vert _{\infty},\epsilon)\right)\leq a\epsilon^{-2p}$, we have that
	\[
	\begin{alignedat}{1} & \log\left(2N(B_{H},\left\Vert \cdot\right\Vert _{\infty},\sqrt{\lambda}\epsilon)\right)\\
	&=\log(2)+\log\left(N(B_{H},\left\Vert \cdot\right\Vert _{\infty},\sqrt{\lambda}\epsilon)\right)\\
	&\leq  \log(2)+a\left(\sqrt{\lambda}\epsilon\right)^{-2p}\\
	&\leq1+a\left(\sqrt{\lambda}\epsilon\right)^{-2p}.
	\end{alignedat}
	\]
	
	Choose $\epsilon=\left(\frac{p}{2}\right)^{\frac{1}{1+p}}\left(\frac{2a}{n}\right)^{\frac{1}{2+2p}}\frac{1}{\sqrt{\lambda}}$ and recall that $a\geq1$.
	Then for $n\geq\frac{p^{2}a}{2}$ we have
	\[
	\begin{alignedat}{1}&\log\left(2N(B_{H},\left\Vert \cdot\right\Vert _{\infty},\sqrt{\lambda}\epsilon)\right)\\
	&\leq1+a\left(\sqrt{\lambda}\epsilon\right)^{-2p}\\ 
	&=  1+a\left(\left(\frac{p}{2}\right)^{\frac{1}{1+p}}\left(\frac{2a}{n}\right)^{\frac{1}{2+2p}}\right)^{-2p}\\& \leq2a\left(\left(\frac{p}{2}\right)^{\frac{1}{1+p}}\left(\frac{2a}{n}\right)^{\frac{1}{2+2p}}\right)^{-2p}\\.
	\end{alignedat}
	\]
	Recall that $B$ is defined by $B=\frac{B_{1}}{2\kappa}+c_{L}\lambda^{-\frac{1}{2}}+1$. Hence, from the assumption on the covering number we have that
	\[
	B\sqrt{\frac{2\log\left(2N(\sqrt{\frac{1}{\lambda}}B_{H},\left\Vert \cdot\right\Vert _{\infty},\epsilon)\right)+2\theta}{n}}\leq \left(\frac{B_{1}}{2\kappa}+c_{L}\lambda^{-\frac{1}{2}}+1\right)\sqrt{\frac{ 4a\left(\left(\frac{p}{2}\right)^{\frac{1}{1+p}}\left(\frac{2a}{n}\right)^{\frac{1}{2+2p}}\right)^{-2p}+2\theta}{n}}
	\]
	and since $\lambda_{n}^{1+p}n\underset{n\rightarrow\infty}{\rightarrow}\infty$, the right hand side of this converges to 0 as $n\rightarrow\infty$.
	Finally, from the choice of $\epsilon$, it follows that both $\frac{2c_{l}\varepsilon}{\kappa}$
	and $4c_{L}\varepsilon$ converge to 0 as $n\rightarrow\infty$. Hence
	for every fixed $\theta,$
	
	\[
	\lambda_{n}\left\Vert f_{D,\lambda_{n}}\right\Vert _{\mathcal{H}}^{2}+R_{L,P}(f_{D,\lambda_{n}})-R_{L,P}^{*}\leq A_{2}(\lambda_{n})+B\sqrt{\frac{2\log\left(2N(\sqrt{\frac{1}{\lambda_{n}}}B_{H},\left\Vert \cdot\right\Vert _{\infty},\epsilon)\right)+2\theta}{n}}+\frac{2c_{l}\varepsilon}{\kappa}+4c_{L}\varepsilon
	\]
	with probability not less than $1-\exp(-\theta)$. The right hand side
	of this converges to 0 as $n\rightarrow\infty$, which implies consistency
	\citep[Lemma 6.5]{steinwart_support_2008}. Since this holds for all
	probability measures $P\in\mathcal{P}$, we obtain $\mathcal{P}$-universal
	consistency.
\end{proof}
	\subsubsection{\label{sub:Proof-of-Lemma 1}Proof of Lemma~\ref{lemma on density}}
	\begin{proof}
		For the sake of completeness, we develop here a finite sample bound on the difference between the kernel density estimator $\hat{g}$ and the true density $g$. While asymptotic results for kernel density estimators are well known in the literature (see, for example, \citealt{silverman1978}), finite sample bounds were not previously studied. In order to develop our bound, we incorporate Bernstein's inequality in our analysis as described below.

		Note that
		\[
		\begin{alignedat}{1} & \frac{1}{n}\sum_{i=1}^n \left|\hat{g}(C_{i})-g(C_{i})\right|\leq\\
		\leq & \frac{1}{n}\sum_{i=1}^n \left|\hat{g}(C_{i})-E\left[\hat{g}(C_{i})\right]\right|+\frac{1}{n}\sum_{i=1}^n \left|E\left[\hat{g}(C_{i})\right]-g(C_{i})\right|\equiv A+B
		\end{alignedat}
		\]

		As in \citet[Proposition 1.1]{tsybakov_introduction_2008}, for any $c_{0}\in\mathcal{Y}$, define
		\[\eta_{i}(c_{0})=K_m\left(\frac{C_{i}-c_{0}}{h}\right)-E_{g}\left[K_m\left(\frac{C_{i}-c_{0}}{h}\right)\right].
		\]
		Then $\eta_{i}(c_{0})$, for $i=1,...,n$ are i.i.d. random variables
		with zero mean and with variance:
		\begin{align*}
		\mathrm{Var}\left[\eta_{i}(c_{0})\right]&=E_{g}\left[\left(\eta_{i}(c_{0})\right)^{2}\right]=E_{g}\left[\left(K_m\left(\frac{C_{i}-c_{0}}{h}\right)-E_{g}\left[K_m\left(\frac{C_{i}-c_{0}}{h}\right)\right]\right)^{2}\right]\leq E_{g}\left[K_m^{2}\left(\frac{C_{i}-c_{0}}{h}\right)\right]
		\\
		&=	 \int_{u}{K_m^{2}\left(\frac{u-c_{0}}{h}\right)g(u)du}\leq g_{max}\int_{u}K_m^{2}\left(\frac{u-c_{0}}{h}\right)du\stackrel{}{=g_{max}h\int_{v}K_m^{2}\left(v\right)dv=D_{1}h}
		\end{align*}
		where the equality before last follows from change of variables and
		where $D_{1}=g_{max}\int_{v}K_m^{2}\left(v\right)dv$. Thus
		$\mathrm{Var}(\hat{g}(c_{0}))=E_{g}\left[\left(\frac{1}{nh}\sum_{i=1}^n \eta_{i}(c_{0})\right)^{2}\right]=\frac{1}{nh^{2}}E_{g}\left[\eta_{1}^{2}(c_{0})\right]\leq\frac{D_{1}h}{nh^{2}}=\frac{D_{1}}{nh}.$
		
		Note that $\mathrm{Var}(\left|\hat{g}(c_{0})-E\left[\hat{g}(c_{0})\right|\right])=E\left[\left(\left|\hat{g}(c_{0})-E\left[\hat{g}(c_{0})\right]\right|\right)^{2}\right]=\mathrm{Var}(\hat{g}(c_{0})-E\left[\hat{g}(c_{0})\right])=\mathrm{Var}(\hat{g}(c_{0}))$.
		Hence $\mathrm{Var}(\left|\hat{g}(c_{0})-E\left[\hat{g}(c_{0})\right|\right])=\mathrm{Var}(\hat{g}(c_{0}))\leq \frac{D_{1}}{nh}$. Using Bernstein's inequality, for any $\theta>0$ we have
		
		\[
		Pr\left(A>\sqrt{\frac{2D_{1}\theta}{n^{2}h}}+\frac{2g_{max}\theta}{3n}\right)\equiv Pr\left(\frac{1}{n}\sum_{i=1}^n \left|\hat{g}(C_{i})-E\left[\hat{g}(C_{i})\right]\right|>\sqrt{\frac{2D_{1}\theta}{n^{2}h}}+\frac{2g_{max}\theta}{3n}\right)\leq \exp(-\theta)
		\]

		For the second term, as in \citet[Proposition 1.2]{tsybakov_introduction_2008},
		we have that
		\[
		B\equiv\frac{1}{n}\sum_{i=1}^n \left|E\left[\hat{g}(C_{i})\right]-g(C_{i})\right|\leq D_{2}h^{\beta}
		\]
		where $D_{2}=\mathcal{L}\left|\pi\right|^{\beta-m}/m!
		\int_{-\infty}^{\infty}{\left|K_m\left(v\right)\right|}\left|v\right|^{\beta}dv<\infty$,
		and for some $\pi\in[0,1].$

		In conclusion, we showed that
		
		\[
		\begin{alignedat}{1} & Pr\left(\frac{1}{n}\sum_{i=1}^n \left|\hat{g}(C_{i})-g(C_{i})\right|>\sqrt{\frac{2D_{1}\theta}{n^{2}h}}+\frac{2g_{max}\theta}{3n}+D_{2}\cdot h^{\beta}\right)\\
		\leq & Pr\left(\frac{1}{n}\sum_{i=1}^n \left|\hat{g}(C_{i})-E\left[\hat{g}(C_{i})\right]\right|+\frac{1}{n}\sum_{i=1}^n \left|E\left[\hat{g}(C_{i})\right]-g(C_{i})\right|>\sqrt{\frac{2D_{1}\theta}{n^{2}h}}+\frac{2g_{max}\theta}{3n}+D_{2}\cdot h^{\beta}\right)\\
		\leq & Pr\left(\frac{1}{n}\sum_{i=1}^n \left|\hat{g}(C_{i})-E\left[\hat{g}(C_{i})\right]\right|+D_{2}\cdot h^{\beta}>\sqrt{\frac{2D_{1}\theta}{n^{2}h}}+\frac{2g_{max}\theta}{3n}+D_{2}\cdot h^{\beta}\right)\\
		= & Pr\left(\frac{1}{n}\sum_{i=1}^n \left|\hat{g}(C_{i})-E\left[\hat{g}(C_{i})\right]\right|>\sqrt{\frac{2D_{1}\theta}{n^{2}h}}+\frac{2g_{max}\theta}{3n}\right)\leq \exp(-\theta)
		\end{alignedat}
		\]
		where $h$ is the bandwidth.
	\end{proof}
	
	\subsubsection{\label{sub:Proof-of-Theorem 2}Proof of Theorem~\ref{theorem g unknown}}
	\begin{proof}
		Note that the proof of this theorem is similar to the proof of of
		Theorem~\ref{ theorem 1 - g known} and thus we will only discuss
		the parts of the proof where they differ. As in Theorem~\ref{ theorem 1 - g known},
		equation \ref{eq:abc},
		\[
		\lambda\left\Vert f_{D,\lambda}\right\Vert _{\mathcal{H}}^{2}+R_{L,P}(f_{D,\lambda})-R_{L,P}^{*}-A_{2}(\lambda)\leq2\left(A_{n}+B_{n}+C_{n}\right)
		\]

		where
		\[
		A_{n}\equiv|R_{L,P}(f)-R_{L,P}(v)|,\,\, B_{n}\equiv|R_{L,P}(v)-R_{L^n,D}(v)|,\,\, \mathrm{and\, where}\,\, C_{n}\equiv|R_{L^n,D}(v)-R_{L^n,D}(f)|,
		\]

		Since $A_{n}$ does not depend on the data-set $D,$ the same bound
		holds as in the proof of Theorem~\ref{ theorem 1 - g known}, that
		is, $A_{n}\leq c_{l}\varepsilon/2\kappa+c_{L}\varepsilon$.
		
		We bound $C_{n}$ as follows:
		
		\begin{alignat*}{1}
		C_{n}\equiv & \left|R_{L^n,D}(v)-R_{L^n,D}(f)\right|\\
		\leq & \left|\frac{1}{n}\sum_{i=1}^n \left[\frac{(1-\Delta_{i})\ell(C_{i},v(Z_{i}))}{\hat{g}(C_{i})}\right]-\frac{1}{n}\sum_{i=1}^n \left[\frac{(1-\Delta_{i})\ell(C_{i},f(Z_{i}))}{\hat{g}(C_{i})}\right]\right|\\
		+ & \left|\frac{1}{n}\sum_{i=1}^n [L(0,v(Z_{i}))]-\frac{1}{n}\sum_{i=1}^n [L(0,f(Z_{i}))]\right|\\
		\equiv & C_{n,1}+C_{n,2}
		\end{alignat*}

		Using the same arguments as in Theorem~\ref{ theorem 1 - g known},
		we can bound $C_{n}$ by $c_{l}\varepsilon/\kappa+c_{L}\varepsilon$.
		Note that the only difference is in the denominator of $C_{n,1}$
		since $g^{-1}\leq (2\kappa)^{-1}$ and $\hat{g}^{-1}\leq\kappa^{-1}$.
		
		Recall that the loss $L(y,f(z))$ is bounded by $B_{2}$. Define $R_{L^n,D,g}(v)$
		by
		\[
		R_{L^n,D,g}(v)=\frac{1}{n}\sum_{i=1}^n \left[\frac{(1-\Delta_{i})\ell(C_{i},v(Z_{i}))}{g(C_{i})}\right]+\frac{1}{n}\sum_{i=1}^n [L(0,v(Z_{i}))].
		\]
		In other words, $R_{L^n,D,g}(v)$ is the empirical risk with the true
		censoring density function $g$.
		
		We bound $B_{n}$ as follows
		
		\begin{alignat*}{1}
		B_{n}= & |R_{L,P}(v)-R_{L^n,D}(v)|\\
		\leq & \left|R_{L,P}(v)-R_{L^n,D,g}(v)\right|+\left|R_{L^n,D,g}(v)-R_{L^n,D}(v)\right|\equiv B_{n,1}+B_{n,2}
		\end{alignat*}

		where
		
		\begin{alignat*}{2}
		\frac{(1-\Delta)\ell(C,v(Z))}{g(C)}+L(0,v(Z)) & \leq\frac{\ell(C,v(Z))}{g(C)}+L(0,v(Z))\leq & \frac{B_{1}}{2\kappa}+B_{2}=B
		\end{alignat*}

		and where
		
		\begin{alignat*}{1}
		B_{n,2}= & \left|R_{L^n,D,g}(v)-R_{L^n,D}(v)\right|=\\
		= & \left|\frac{1}{n}\sum_{i=1}^n \left[\frac{(1-\Delta_{i})\ell(C_{i},v(Z_{i}))}{g(C_{i})}\right]-\frac{1}{n}\sum_{i=1}^n \left[\frac{(1-\Delta_{i})\ell(C_{i},v(Z_{i}))}{\hat{g}(C_{i})}\right]\right|\\
		= & \left|\frac{1}{n}\sum_{i=1}^n \left[(1-\Delta_{i})\ell(C_{i},v(Z_{i}))\left(\frac{1}{g(C_{i})}-\frac{1}{\hat{g}(C_{i})}\right)\right]\right|\\
		\leq & \frac{1}{n}\sum_{i=1}^n \left[\left|\ell(C_{i},v(Z_{i}))\left(\frac{1}{g(C_{i})}-\frac{1}{\hat{g}(C_{i})}\right)\right|\right]\\
		= & \frac{B_{1}}{n}\sum_{i=1}^n \left[\left|\frac{\hat{g}(C_{i})-g(C_{i})}{g(C_{i})\hat{g}(C_{i})}\right|\right]\leq\frac{B_{1}}{2\kappa^{2}n}\sum_{i=1}^n \left[\left|\hat{g}(C_{i})-g(C_{i})\right|\right].
		\end{alignat*}

		Note that these inequalities hold for all functions $v\in\mathcal{F_{\varepsilon}}\subseteq\lambda^{-\nicefrac{1}{2}}B_{H}$.
		We would like to bound the last expression using Lemma~\ref{lemma on density}.
		Let \[\eta=\frac{B_{1}}{2\kappa^{2}}\left(\sqrt{\frac{2D_{1}\theta}{n^{2}h}}+\frac{2g_{max}\theta}{3n}+D_{2}\cdot h^{\beta}\right),
		\]
		then by Lemma~\ref{lemma on density}
		
		\[
		\begin{alignedat}{1} & Pr(B_{n,2}>\eta)\leq Pr\left(\frac{B_{1}}{2\kappa^{2}n}\sum_{i=1}^n \left[\left|\hat{g}(C_{i})-g(C_{i})\right|\right]>\eta\right)\\
		= & Pr\left(\frac{B_{1}}{2\kappa^{2}n}\sum_{i=1}^n \left[\left|\hat{g}(C_{i})-g(C_{i})\right|\right]>\frac{B_{1}}{2\kappa^{2}}\left(\sqrt{\frac{2D_{1}\theta}{n^{2}h}}+\frac{2g_{max}\theta}{3n}+D_{2}\cdot h^{\beta}\right)\right)\\
		= & Pr\left(\frac{1}{n}\sum_{i=1}^n \left[\left|\hat{g}(C_{i})-g(C_{i})\right|\right]>\left(\sqrt{\frac{2D_{1}\theta}{n^{2}h}}+\frac{2g_{max}\theta}{3n}+D_{2}\cdot h^{\beta}\right)\right)\\
		\leq & \exp(-\theta).
		\end{alignedat}
		\]

		We need to bound the term $B_{n,1}(v)\equiv\left|R_{L,P}(v)-R_{L^n,D,g}(v)\right|$.
		By the union bound, for all $\mu>0$

		\begin{align*}
		Pr\left(\underset{v\in\mathcal{F_{\varepsilon}}}{\sup}B_{n,1}(v)\geq B\sqrt{\frac{\mu}{2n}}\right)&=Pr\left(\underset{v\in\mathcal{F_{\varepsilon}}}{\sup}\left|R_{L,P}(v)-R_{L^n,D,g}(v)\right|
		\geq B\sqrt{\frac{\mu}{2n}}\right)
		\\
		&\leq\underset{v\in\mathcal{F_{\varepsilon}}}{\sum}Pr\left(|R_{L,P}(v)-R_{L^n,D,g}(v)|\geq B\sqrt{\frac{\mu}{2n}}\right).
		\end{align*}

		We showed that $(1-\Delta)\ell(C,v(Z))/g(C)+L(0,v(Z))\leq B$. Note also that $R_{L,P}(v)=R_{L^{n},P}(v)=R_{L^{n},P,g}(v)$; That is, $R_{L,P}(v)$ is the expectation of $R_{L^{n},D,g}(v)$.
		Hence by Hoeffding's inequality, the last term can then be bounded
		again by $2|\mathcal{F}_{\varepsilon}|\exp(-\mu)$, where $\mathcal{F_{\epsilon}}$
		is an $\varepsilon$-net of $\lambda^{-\nicefrac{1}{2}}\overline{B_{H}}$
		with cardinality
		\[
		|\mathcal{F}\varepsilon|=N\left(\lambda^{-\nicefrac{1}{2}}B_{H},\left\Vert \cdot\right\Vert _{\infty},\epsilon\right)<\infty.
		\]

		Define $\mu=\log(2|\mathcal{F_{\varepsilon}}|)+\theta$, then
		\[
		Pr\left(\underset{v\in\mathcal{F_{\varepsilon}}}{\sup}B_{n,1}(v)\geq B\sqrt{\frac{ln(2|\mathcal{F_{\varepsilon}}|)+\theta}{2n}}\right)\leq \exp(-\theta)
		\]
		In conclusion we have that
		
		\begin{alignat*}{1}
		Pr & \left(\lambda\left\Vert f_{D,\lambda}\right\Vert _{\mathcal{H}}^{2}+R_{L,P}(f_{D,\lambda})-R_{L,P}^{*}-A_{2}(\lambda)\geq B\sqrt{\frac{2\mu}{n}}+\frac{3c_{l}\varepsilon}{\kappa}+4c_{L}\varepsilon+2\eta\right)\\
		\leq & Pr\left(2\underset{\left\Vert f\right\Vert _{\mathcal{H}}\leq\lambda^{-\nicefrac{1}{2}}}{sup}|R_{L,P}(f)-R_{L^n,D}(f)|\geq B\sqrt{\frac{2\mu}{n}}+\frac{3c_{l}\varepsilon}{\kappa}+4c_{L}\varepsilon+2\eta\right)\\
		\leq & Pr\left(2\left(\underset{v\in\mathcal{F_{\varepsilon}}}{sup}|R_{L,P}(v)-R_{L^n,D}(v)|+\frac{3}{2\kappa}c_{l}\varepsilon+2c_{L}\varepsilon\right)\geq B\sqrt{\frac{2\mu}{n}}+\frac{3c_{l}\varepsilon}{\kappa}+4c_{L}\varepsilon+2\eta\right)\\
		\leq & Pr\left(2\left(\underset{v\in\mathcal{F_{\varepsilon}}}{\sup}B_{n,1}(v)+B_{n,2}(v)\right)\geq B\sqrt{\frac{2\mu}{n}}+2\eta\right)\\
		\leq & Pr\left(\underset{v\in\mathcal{F_{\varepsilon}}}{\sup}B_{n,1}(v)+B_{n,2}(v)\geq B\sqrt{\frac{\mu}{2n}}+\eta\right)\\
		\leq & Pr\left(\underset{v\in\mathcal{F_{\varepsilon}}}{\sup}B_{n,1}\geq B\sqrt{\frac{ln(2|\mathcal{F_{\varepsilon}}|)+\theta}{2n}}\right)+Pr\left(\underset{v\in\mathcal{F_{\varepsilon}}}{\sup}B_{n,2}(v)\geq\eta\right)\\
		\leq & \exp(-\theta)+\exp(-\theta)=2\exp(-\theta)
		\end{alignat*}

		and the result follows.
	\end{proof}

\subsubsection{\label{sub:Proof-of-cor 2}Proof of Corollary~\ref{consistency 2}}
\begin{proof}
	\label{proof consistency2} 
	Note that the only difference between Corollary~\ref{consistency 2} and Corollary~\ref{cor - consistency} is in the term $2\eta$. Recall that $\eta$ is defined by $\eta\equiv\frac{B_{1}}{2\kappa^{2}}\left(\sqrt{\frac{2D_{1}\theta}{n^{2}h}}+\frac{2g_{max}\theta}{3n}+D_{2}\cdot h^{\beta}\right)$. Choose $h$ such that $h\underset{n\rightarrow\infty}{\rightarrow}0$ and that $h^{0.5}n\underset{n\rightarrow\infty}{\rightarrow}\infty$. Then $\eta\underset{n\rightarrow\infty}{\rightarrow}0$. 
	Choose $\lambda=\lambda_{n}$ and $\epsilon=\left(\frac{p}{2}\right)^{\frac{1}{1+p}}\left(\frac{2a}{n}\right)^{\frac{1}{2+2p}}\frac{1}{\sqrt{\lambda}}$.
	Then as in Corollary~\ref{cor - consistency}, all other terms converge to zero as ${n\rightarrow\infty}$ which implies consistency \citep[Lemma 6.5]{steinwart_support_2008}. Since this holds for all
	probability measures $P\in\mathcal{P}$, we obtain $\mathcal{P}$-universal
	consistency.
\end{proof}

\subsection{\label{Learning rates}Learning Rates}

In this subsection we derive learning rates for cases I and II.
\begin{defn}
	\label{def learning rate}A learning method is said to learn with
	rate $\epsilon_{n}\subset(0,1]$ that converges to zero if for all
	$n\geq1$ and all $\tau\in(0,1]$, $Pr\left(R_{L,P}(f_{D})-R_{L,P}^{*}\leq c_{P}c_{\tau}\epsilon_{n}\right)\geq1-\tau$
	, where $c_{\tau}$ and $c_{P}$ are constants such that $c_{\tau}\in[1,\infty)$
	and $c_{P}>0$.\end{defn}
	
	We demonstrate how to derive learning rates from the same oracle inequalities used for the consistency proofs. While faster learning rates can be achieved under further assumptions in a similar manner, they further complicate the calculations and are beyond the scope of this paper.
	
\begin{thm}
	\label{theorem learning rate}Assume that (A\ref{as:A1})-(A\ref{as:A4}) hold. Choose $0<\lambda_{n}<1$ and assume that there exist constants $a\geq1,\, p>0$ such that $\log\left(N(B_{H},\left\Vert \cdot\right\Vert _{\infty},\epsilon)\right)\leq a\epsilon^{-2p}$.
	Additionally, assume that there exist constants $c>0,\,\gamma\in(0,1]$
	such that $A_{2}(\lambda)\leq c\lambda^{\gamma}$. Then
	
	(i) If $g$ is known, the learning rate is given by $n^{-\frac{\gamma}{(1+p)(2\gamma+1)}}$.
	
	(ii) If $g$ is not known and the setup of Theorem~\ref{theorem g unknown}
	holds, then the leraning rate is given by $n^{-\min\left(\frac{\gamma}{(1+p)(2\gamma+1)},\frac{2\beta}{2\beta+1}\right)}$.
\end{thm}

		\subsubsection{\label{sub:Proof-of-Theorem 3}Proof of Theorem~\ref{theorem learning rate}}
	\begin{proof}
		\textbf{\large{}Case I}{\large \par}
		
		By Theorem~\ref{ theorem 1 - g known},
		
		\[
		\lambda\left\Vert f_{D,\lambda}\right\Vert _{\mathcal{H}}^{2}+R_{L,P}(f_{D,\lambda})-R_{L,P}^{*}-A_{2}(\lambda)\leq B\sqrt{\frac{2\log\left(2N(\lambda^{-\nicefrac{1}{2}}B_{H},\left\Vert \cdot\right\Vert _{\infty},\epsilon)\right)+2\theta}{n}}+\frac{2c_{l}\varepsilon}{\kappa}+4c_{L}\varepsilon
		\]
		with probability not less than $1-\exp(-\theta)$. For any compact set
		$\mathcal{S}=[-S,S]\subset\mathbb{R}$, Both $L$ and $l$ are bounded
		and Lipschitz continuous with Lipschitz constants $c_{L}\leq 2\tau^{-2}(S+\tau)$
		and $c_{l}=2\tau^{-2}$. Hence,
		
		\begin{equation}
		\begin{alignedat}{1} & \lambda\left\Vert f_{D,\lambda}\right\Vert _{\mathcal{H}}^{2}+R_{L,P}(f_{D,\lambda})-R_{L,P}^{*}-A_{2}(\lambda)\\
		\leq & B\sqrt{\frac{2\log\left(2N(B_{H},\left\Vert \cdot\right\Vert _{\infty},\sqrt{\lambda}\epsilon)\right)+2\theta}{n}}+\frac{2c_{l}\varepsilon}{\kappa}+4c_{L}\varepsilon\\
		\leq & B\sqrt{\frac{2\log\left(2N(B_{H},\left\Vert \cdot\right\Vert _{\infty},\sqrt{\lambda}\epsilon)\right)+2\theta}{n}}+\frac{4\varepsilon}{\kappa\tau^{2}}+\frac{8(S+\tau)}{\tau^{2}}\varepsilon\\
		= & B\sqrt{\frac{2\log\left(2N(B_{H},\left\Vert \cdot\right\Vert _{\infty},\sqrt{\lambda}\epsilon)\right)+2\theta}{n}}+M\cdot\epsilon
		\end{alignedat}
		\label{eq:M-1}
		\end{equation}

		where $M=4\tau^{-2}\left(\kappa^{-1}+2(S+\tau)\right)$.
		
		By the assumption $\log\left(N(B_{H},\left\Vert \cdot\right\Vert _{\infty},\epsilon)\right)\leq a\epsilon^{-2p}$, we have that:
		\[
		\begin{alignedat}{1} & \log\left(2N(B_{H},\left\Vert \cdot\right\Vert _{\infty},\sqrt{\lambda}\epsilon)\right)=\log(2)+\log\left(N(B_{H},\left\Vert \cdot\right\Vert _{\infty},\sqrt{\lambda}\epsilon)\right)\\
		\leq & \log(2)+a\left(\sqrt{\lambda}\epsilon\right)^{-2p}\leq2a\left(\sqrt{\lambda}\epsilon\right)^{-2p}.
		\end{alignedat}
		\]

		Choose $\epsilon=\left(\frac{p}{2}\right)^{\frac{1}{1+p}}\left(\frac{2a}{n}\right)^{\frac{1}{2+2p}}\frac{1}{\sqrt{\lambda}}$.
		Then
		\begin{equation}
		\begin{alignedat}{1} & a\left(\sqrt{\lambda}\epsilon\right)^{-2p}\\
		= & a\left(\left(\frac{p}{2}\right)^{\frac{1}{1+p}}\left(\frac{2a}{n}\right)^{\frac{1}{2+2p}}\right)^{-2p}.
		\end{alignedat}
		\label{eq:covering}
		\end{equation}

		By (\ref{eq:M-1}) and (\ref{eq:covering}),
		\begin{equation}
		\begin{alignedat}{1} & \lambda\left\Vert f_{D,\lambda}\right\Vert _{\mathcal{H}}^{2}+R_{L,P}(f_{D,\lambda})-R_{L,P}^{*}-A_{2}(\lambda)\\
		\leq & B\sqrt{\frac{4a\left(\left(\frac{p}{2}\right)^{\frac{1}{1+p}}\left(\frac{2a}{n}\right)^{\frac{1}{2+2p}}\right)^{-2p}+2\theta}{n}}+M\left(\frac{p}{2}\right)^{\frac{1}{1+p}}\left(\frac{2a}{n}\right)^{\frac{1}{2+2p}}\frac{1}{\sqrt{\lambda}}\\
		\leq & B\left(\sqrt{\frac{4a\left(\left(\frac{p}{2}\right)^{\frac{1}{1+p}}\left(\frac{2a}{n}\right)^{\frac{1}{2+2p}}\right)^{-2p}}{n}}+\sqrt{\frac{2\theta}{n}}\right)+\frac{M}{\sqrt{\lambda}}\left(\frac{p}{2}\right)^{\frac{1}{1+p}}\left(\frac{2a}{n}\right)^{\frac{1}{2+2p}}\\
		= & B\left(\frac{\sqrt{4a}\left(\left(\frac{p}{2}\right)^{\frac{-p}{1+p}}\left(\frac{2a}{n}\right)^{\frac{-p}{2+2p}}\right)}{\sqrt{n}}\right)+\frac{M}{\sqrt{\lambda}}\left(\frac{p}{2}\right)^{\frac{1}{1+p}}\left(\frac{2a}{n}\right)^{\frac{1}{2+2p}}+B\sqrt{\frac{2\theta}{n}}\\
		= & \left(\frac{p}{2}\right)^{\frac{-p}{1+p}}\left[B\sqrt{2}\left(\frac{2a}{n}\right)^{\frac{1}{2+2p}}+\frac{M}{\sqrt{\lambda}}\frac{p}{2}\left(\frac{2a}{n}\right)^{\frac{1}{2+2p}}+\right]+B\sqrt{\frac{2\theta}{n}}
		\end{alignedat}
		\label{eq:leraring rate1}
		\end{equation}

		Recall that $B_{2}=c_{L}\lambda^{-\nicefrac{1}{2}}+1$ and $B=B_{1}/2\kappa+B_{2}$,
		where $B_{1}$ is some bound on the derivative of the loss. Since
		$0<\lambda<1$, then $1<\lambda^{-\nicefrac{1}{2}}$, and therefor \[B_{2}\leq c_{L}\lambda^{-\nicefrac{1}{2}}+\lambda^{-\nicefrac{1}{2}}=\lambda^{-\nicefrac{1}{2}}(c_{L}+1)\leq\lambda^{-\nicefrac{1}{2}}\left(\frac{2(S+\tau)}{\tau^{2}}+1\right).
		\]
		Earlier we defined $M$ such that $\kappa=4/M\tau^{2}-8(S+\tau)$.
		Thus,
		\[
		B\leq\frac{B_{1}}{2\kappa}+\frac{1}{\sqrt{\lambda}}\left(\frac{2(S+\tau)+\tau^{2}}{\tau^{2}}\right)=\frac{B_{1}(M\tau^{2}-8(S+\tau))}{8}+\frac{1}{\sqrt{\lambda}}\left(\frac{2(S+\tau)+\tau^{2}}{\tau^{2}}\right)=
		\]
		\[
		=\frac{\sqrt{\lambda}B_{1}(M\tau^{2}-8(S+\tau))+8\left(\frac{2(S+\tau)+\tau^{2}}{\tau^{2}}\right)}{8\sqrt{\lambda}}\leq\frac{B_{1}(M\tau^{2})+8+16\left(\frac{S+\tau}{\tau^{2}}\right)}{8\sqrt{\lambda}}\equiv\frac{N}{\sqrt{\lambda}},
		\]
		where we define $N\equiv 8^{-1}\left(B_{1}M\tau^{2}+8+16\tau^{-2}\left(S+\tau\right)\right)$.
		
		Hence we can bound (\ref{eq:leraring rate1}) by
		
		\[
		\begin{alignedat}{1} & \left(\frac{p}{2}\right)^{\frac{-p}{1+p}}\left[\frac{\sqrt{2}N}{\sqrt{\lambda}}\left(\frac{2a}{n}\right)^{\frac{1}{2+2p}}+\frac{M}{\sqrt{\lambda}}\frac{p}{2}\left(\frac{2a}{n}\right)^{\frac{1}{2+2p}}\right]+\frac{N}{\sqrt{\lambda}}\sqrt{\frac{2\theta}{n}}\\
		\leq & \left(\frac{p}{2}\right)^{\frac{-p}{1+p}}\frac{N}{\sqrt{\lambda}}\left[\sqrt{2}\left(\frac{2a}{n}\right)^{\frac{1}{2+2p}}+\frac{Mp}{2N}\left(\frac{2a}{n}\right)^{\frac{1}{2+2p}}\right]+\frac{N}{\sqrt{\lambda}}\sqrt{\frac{2\theta}{n}}\\
		\leq & \left(\frac{p}{2}\right)^{\frac{-p}{1+p}}\frac{N}{\sqrt{\lambda}}\left[2\left(\frac{2a}{n}\right)^{\frac{1}{2+2p}}+\frac{Mp}{N}\left(\frac{2a}{n}\right)^{\frac{1}{2+2p}}\right]+\frac{N}{\sqrt{\lambda}}\sqrt{\frac{2\theta}{n}}
		\end{alignedat}
		\]

		Choose \[B_{1}\geq\frac{4}{\tau^{2}}-\left(2+4\left(\frac{S+\tau}{\tau^{2}}\right)\right)\left(\frac{1}{\kappa}+2S+2\tau\right)^{-1}.
		\]
		Note that \[M=\frac{4}{\tau^{2}}\left(\frac{1}{\kappa}+2(S+\tau)\right)\leq\frac{B_{1}(M\tau^{2})}{4}+2+4\left(\frac{S+\tau}{\tau^{2}}\right)=2N.
		\]
		Consequently, for our choice of $B_{1}$,
		we have that $M\leq2N$ or $M/2N\leq1$. Note also that $(p+1)(2/p)^{\nicefrac{p}{1+p}}\leq3,$
		hence:
		
		\[
		\begin{alignedat}{1}\left(\frac{p}{2}\right)^{\frac{-p}{1+p}}\frac{N}{\sqrt{\lambda}}\left(\frac{2a}{n}\right)^{\frac{1}{2+2p}}\left(2+\frac{M}{N}p\right)+\frac{N}{\sqrt{\lambda}}\sqrt{\frac{2\theta}{n}} & \leq\left(\frac{p}{2}\right)^{\frac{-p}{1+p}}\left(p+1\right)2\frac{N}{\sqrt{\lambda}}\left(\frac{2a}{n}\right)^{\frac{1}{2+2p}}+\frac{N}{\sqrt{\lambda}}\sqrt{\frac{2\theta}{n}}\\
		& \leq\frac{N}{\sqrt{\lambda}}\left[6\left(\frac{2a}{n}\right)^{\frac{1}{2+2p}}+\sqrt{\frac{2\theta}{n}}\right].
		\end{alignedat}
		\]

		Since $A_{2}(\lambda)\leq c\lambda^{\gamma}$ for constants $c>0,$
		and $\gamma\in(0,1]$,
		
		\begin{equation}
		\lambda\left\Vert f_{D,\lambda}\right\Vert _{\mathcal{H}}^{2}+R_{L,P}(f_{D,\lambda})-R_{L,P}^{*}\leq c\lambda^{\gamma}+\frac{N}{\sqrt{\lambda}}\left[6\left(\frac{2a}{n}\right)^{\frac{1}{2+2p}}+\sqrt{\frac{2\theta}{n}}\right]\label{eq:learning rate bound}
		\end{equation}

		We would like to choose a sequence $\lambda_{n}$ that will minimize
		the bound in (\ref{eq:learning rate bound}). 
		
		Define \[W(\lambda)=c\lambda^{\gamma}+\frac{N}{\sqrt{\lambda}}\left[6\left(\frac{2a}{n}\right)^{\frac{1}{2+2p}}+\sqrt{\frac{2\theta}{n}}\right].
		\]
		Differentiating $W$ with respect to $\lambda$ and setting to zero
		yields:
		
		\[
		\begin{alignedat}{1}\frac{dW(\lambda)}{d\lambda}= & c\gamma\lambda^{\gamma-1}-\frac{1}{2}N\lambda^{-\frac{3}{2}}\left[6\left(\frac{2a}{n}\right)^{\frac{1}{2+2p}}+\sqrt{\frac{2\theta}{n}}\right]=0\\
		\Leftrightarrow\\
		c\gamma\lambda^{\gamma-1}= & \frac{1}{2}N\lambda^{-\frac{3}{2}}\left[6\left(\frac{2a}{n}\right)^{\frac{1}{2+2p}}+\sqrt{\frac{2\theta}{n}}\right]\\
		\Leftrightarrow\lambda= & \left(\frac{1}{2c\gamma}N\left[6\left(\frac{2a}{n}\right)^{\frac{1}{2+2p}}+\sqrt{\frac{2\theta}{n}}\right]\right)^{\frac{1}{\gamma+\frac{1}{2}}}\propto\left(\frac{1}{n}^{\frac{1}{2+2p}}+\left(\frac{1}{n}\right)^{\frac{1}{2}}\right)^{\frac{2}{2\gamma+1}}\\
		\Rightarrow\lambda\propto & n^{-\frac{1}{(1+p)(2\gamma+1)}}
		\end{alignedat}
		\]

		Since the second derivative of $W$ (with respect to $\lambda)$ is
		positive, $\lambda$ is the minimizer. by (\ref{eq:learning rate bound}),
		
		\begin{equation}
		\begin{alignedat}{1} & Pr\left(R_{L,P}(f_{D,\lambda})-R_{L,P}^{*}\leq c\lambda^{\gamma}+\frac{N}{\sqrt{\lambda}}\left[6\left(\frac{2a}{n}\right)^{\frac{1}{2+2p}}+\sqrt{\frac{2\theta}{n}}\right]\right)\geq1-\exp(-\theta)\end{alignedat}
		.\label{eq:learning rate prob}
		\end{equation}

		By the choice of $\lambda_{n},$ the bound in equation (\ref{eq:learning rate prob})
		can be written as
		
		\[
		\begin{alignedat}{1} & cn^{-\frac{\gamma}{(1+p)(2\gamma+1)}}+Nn^{\frac{1}{2(1+p)(2\gamma+1)}}\left[6\left(2a\right)^{\frac{1}{2+2p}}n^{-\frac{1}{2+2p}}+\left(2\theta\right)^{\frac{1}{2}}n^{-\frac{1}{2}}\right]\\
		= & cn^{-\frac{\gamma}{(1+p)(2\gamma+1)}}+N\cdot6\left(2a\right)^{\frac{1}{2+2p}}n^{-\frac{\gamma}{(1+p)(2\gamma+1)}}+N\left(2\theta\right)^{\frac{1}{2}}n^{-\frac{2\gamma(1+p)+p}{2(1+p)(2\gamma+1)}}\\
		\leq & cn^{-\frac{\gamma}{(1+p)(2\gamma+1)}}+N\cdot6\left(2a\right)^{\frac{1}{2+2p}}n^{-\frac{\gamma}{(1+p)(2\gamma+1)}}+N\left(2\theta\right)^{\frac{1}{2}}n^{-\frac{\gamma}{(1+p)(2\gamma+1)}}\\
		= & n^{-\frac{\gamma}{(1+p)(2\gamma+1)}}\left(c+N\cdot6\left(2a\right)^{\frac{1}{2+2p}}+N\left(2\theta\right)^{\frac{1}{2}}\right)\\
		\leq & Q(1+\sqrt{\theta})n^{-\frac{\gamma}{(1+p)(2\gamma+1)}}
		\end{alignedat}
		\]

		where $Q$ is a constant that does not depend on $n$ or on $\theta$.
		
		In conclusion, by choosing a sequence $\lambda_{n}$ that behaves
		like $n^{-\nicefrac{1}{(1+p)(2\gamma+1)}}$, we have that the resulting
		learning rate is given by
		\[
		Pr\left(R_{L,P}(f_{D,\lambda})-R_{L,P}^{*}\leq Q(1+\sqrt{\theta})n^{-\frac{\gamma}{(1+p)(2\gamma+1)}}\right)\geq1-\exp(-\theta).
		\]

		\textbf{\large{}Case II}{\large \par}
		By Theorem~\ref{theorem g unknown},
		\[
		\lambda\left\Vert f_{D,\lambda}\right\Vert _{\mathcal{H}}^{2}+R_{L,P}(f_{D,\lambda})-R_{L,P}^{*}-A_{2}(\lambda)\geq B\sqrt{\frac{2\log\left(2N(\lambda^{-\nicefrac{1}{2}}B_{H},\left\Vert \cdot\right\Vert _{\infty},\epsilon)\right)+2\theta}{n}}+\frac{3c_{l}\varepsilon}{\kappa}+4c_{L}\varepsilon+2\eta
		\]
	with probability not greater than $2\exp(-\theta)$ and where \[\eta\equiv\frac{B_{1}}{2\kappa^{2}}\left(\sqrt{\frac{2D_{1}\theta}{n^{2}h}}+\frac{2g_{max}\theta}{3n}+D_{2}\cdot h^{\beta}\right).
	\]
		
		Choose \[\epsilon=\left(\frac{p}{2}\right)^{\frac{1}{1+p}}\left(\frac{2a}{n}\right)^{\frac{1}{2+2p}}\frac{1}{\sqrt{\lambda}},
		\]
		\[M=\frac{2}{\tau^{2}}\left(\frac{3}{\kappa}+4(S+\tau)\right),
		\] \[B_{1}\geq\frac{6}{\tau^{2}}-\left(6+12\left(\frac{S+\tau}{\tau^{2}}\right)\right)\left(\frac{3}{\kappa}+4S+4\tau\right)^{-1},
		\]
		and define $N=12^{-1}\left(B_{1}M\tau^{2}+12+24\tau^{-2}\left(S+\tau\right)\right)$
		, then as in (\ref{eq:learning rate bound}), a very similar calculation
		shows that
		\[
		\begin{alignedat}{1} & \lambda\left\Vert f_{D,\lambda}\right\Vert _{\mathcal{H}}^{2}+R_{L,P}(f_{D,\lambda})-R_{L,P}^{*}\leq c\lambda^{\gamma}+\frac{N}{\sqrt{\lambda}}\left[6\left(\frac{2a}{n}\right)^{\frac{1}{2+2p}}+\sqrt{\frac{2\theta}{n}}\right]+2\eta.\end{alignedat}
		\]
		
		We would like to choose the bandwidth $h$ that minimizes $\eta$. The minimum is achieved at $h^{*}$ where \[h^{*}=\left(\frac{D_1 \theta}{2D_2^2\beta^2n^2}\right)^{\nicefrac{1}{2\beta+1}}.
		\] 
		 Substituting this result into $\eta$ yields
		\[
		\eta=\frac{B_{1}}{2\kappa^{2}}\left(\sqrt{2D_1}\left(\frac{D_{1}}{{2D_{2}}^2\beta^2}\right)^{\frac{-1}{2(2\beta+1)}}n^{-\frac{2\beta}{2\beta+1}}\theta^{\frac{\beta}{2\beta+1}}+\frac{2g_{max}\theta}{3n}+D_{2}\left(\frac{D_{1}}{2D_{2}^2\beta^2}\right)^{\frac{\beta}{2\beta+1}}\theta^{\frac{\beta}{2\beta+1}}n^{-\frac{2\beta}{2\beta+1}}\right)
		\]
		or
		\[
		\eta=\tilde{D}n^{-\min\left(1,\frac{2\beta}{2\beta+1}\right)}\max\left(\theta,\theta^{\frac{\beta}{2\beta+1}}\right)=\tilde{D}n^{-\frac{2\beta}{2\beta+1}}\max\left(\theta,\theta^{\frac{\beta}{2\beta+1}}\right)
		\]
		Where $\tilde{D}$ is a constant that does not depend on $\theta$ or on $n$.
		
		Hence,
		
		\[
		\begin{alignedat}{1} & \lambda\left\Vert f_{D,\lambda}\right\Vert _{\mathcal{H}}^{2}+R_{L,P}(f_{D,\lambda})-R_{L,P}^{*}\leq c\lambda^{\gamma}+\frac{N}{\sqrt{\lambda}}\left[6\left(\frac{2a}{n}\right)^{\frac{1}{2+2p}}+\sqrt{\frac{2\theta}{n}}\right]+2\eta\\
		\leq & c\lambda^{\gamma}+\frac{N}{\sqrt{\lambda}}\left[6\left(\frac{2a}{n}\right)^{\frac{1}{2+2p}}+\sqrt{\frac{2\theta}{n}}\right]+\tilde{D}n^{-\frac{2\beta}{2\beta+1}}\max\left(\theta,\theta^{\frac{\beta}{2\beta+1}}\right)
		\end{alignedat}
		\]
		Similarly to Case I, choosing $\lambda_{n}\propto n^{-\frac{1}{(1+p)(2\gamma+1)}}$
		minimizes the last bound (note that the choice of $\lambda_{n}$ does
		not depend on $\eta$). Hence the resulting learning rate is
		given by
		
		\[
		Pr(D\in(\mathcal{Z\times\mathcal{Y}})^{n}\,:\,\mathcal{R}_{L,P}(f_{D,\lambda_{n}})-\mathcal{R}_{L,P}^{*}\leq Q \max\left(\theta,1+\sqrt{\theta}\right)n^{-min\left(\frac{\gamma}{(1+p)(2\gamma+1)},\frac{2\beta}{2\beta+1}\right)})\geq1-\exp(-\theta)
		\]

		where $Q$ is a constant that does not depend on $n$ or on $\theta$.
	\end{proof}
\bibliographystyle{spbasic}
\bibliography{KM_CSD_bib_author_initials}

\end{document}